\documentclass{amsart}

\usepackage{amsmath}
\usepackage{amssymb}
\usepackage{mathrsfs}
\usepackage{stmaryrd}

\usepackage{hyperref}
\usepackage{enumitem}

\usepackage[british]{babel}
\usepackage[all]{xy}
\usepackage{comment}
\usepackage{graphicx}
\usepackage{caption}
\usepackage{subcaption}
\usepackage{xcolor}
\usepackage{tikz}
\usetikzlibrary{cd}

\usepackage[abbrev]{amsrefs}

\newtheorem{theorem}{Theorem}[section]
\newtheorem{lemma}[theorem]{Lemma}
\newtheorem{prop}[theorem]{Proposition}
\newtheorem{cor}[theorem]{Corollary}

\theoremstyle{definition}
\newtheorem{remark}[theorem]{Remark}

\numberwithin{equation}{section}

\newcommand{\Kps}[2]{M^\mathrm{Kps}_{#1, #2}}
\newcommand{\Kss}[2]{\mathcal{M}^\mathrm{Kss}_{#1, #2}}
\newcommand{\modulismooth}{M_{\rho=4}^\mathrm{sm}}

\renewcommand{\setminus}{\smallsetminus}
\renewcommand{\emptyset}{\varnothing}
\newcommand{\into}{\hookrightarrow} 
\newcommand{\onto}{\twoheadrightarrow} 

\newcommand{\al}{\alpha}

\DeclareMathOperator{\Hom}{Hom} 
\DeclareMathOperator{\Spec}{Spec} 
\DeclareMathOperator{\Spf}{Spf} 
\DeclareMathOperator{\Pic}{Pic}
\DeclareMathOperator{\Cl}{Cl}
\DeclareMathOperator{\Div}{Div}
\DeclareMathOperator{\conv}{conv}
\DeclareMathOperator{\Aut}{Aut}

\DeclareMathOperator{\GL}{GL}
\newcommand{\git}{/ \! \!  /}

\def\pow#1{ \llbracket  #1 \rrbracket }

\newcommand{\Def}[1]{\mathrm{Def}_{#1}}

\newcommand\cO{\mathcal{O}}
\newcommand\cX{\mathcal{X}}
\newcommand\cU{\mathcal{U}}

\renewcommand\AA{\mathbb{A}}
\newcommand\CC{\mathbb{C}}

\newcommand\NN{\mathbb{N}}
\newcommand\PP{\mathbb{P}}
\newcommand\QQ{\mathbb{Q}}
\newcommand\RR{\mathbb{R}}
\newcommand\TT{\mathbb{T}}
\newcommand\ZZ{\mathbb{Z}}

\newcommand\rH{\mathrm{H}}

\title[On K-moduli of Fano threefolds with degree 28 and Picard rank 4]{On K-moduli of Fano threefolds \\ with degree 28 and Picard rank 4}

\author{Liana Heuberger}
\address{Institut de Mathématiques de Marseille (I2M), 3 place Victor Hugo, 13331 Marseille cedex 3, France}
\email{liana.heuberger@univ-amu.fr}

\author{Andrea Petracci}
\address{Dipartimento di Matematica, Universit\`a di Bologna, Piazza di Porta San~Donato 5, Bologna, 40126, Italy}
\email{a.petracci@unibo.it}

\begin{document}

\begin{abstract}
We analyse the local structure of the K-moduli space of Fano varieties at a toric singular K-polystable Fano $3$-fold, which deforms to smooth Fano $3$-folds with anticanonical volume $28$ and Picard rank $4$. In particular, by constructing an algebraic deformation of this toric singular Fano, we show that the irreducible component of K-moduli parametrising these smooth Fano $3$-folds is a rational surface.
\end{abstract}

\maketitle

\section{Introduction}

\subsection{Context}

We work over $\CC$.  In the study of Fano varieties \emph{K-stability} \cite{tian_KE, donaldson_stability}  has become a fundamental topic in recent years, for two main reasons: it is the algebraic counterpart of the existence of K\"ahler--Einstein metrics \cite{chen_donaldson_sun,tian_KstabKE}, and it has allowed to construct projective moduli spaces for Fano varieties.

More precisely, by \cite{ABHLX,BLX,properness_K_moduli,blum_xu_uniqueness,jiang_boundedness,lwx,projectivity_K_moduli_final,xu_minimizing,xu_zhuang}, for every positive integer $n$ and for every positive rational number $v$, there exists a projective scheme $\Kps{n}{v}$ over $\CC$, whose closed points are in natural one-to-one correspondence with K-polystable Fano varieties $X$ of dimension $n$ and anticanonical volume $v$. This is called a \emph{K-moduli space}. We refer the reader to \cite{xu_survey} for a survey on this topic.

An important problem is to decide if a given Fano variety is K-(semi/poly)stable: this is the \emph{Calabi problem}.
For smooth del Pezzo surfaces this has been completely solved in \cite{tian_del_pezzo}.
For the general member of the 105 deformation families of smooth Fano $3$-folds, the Calabi problem has been solved in \cite{calabi_problem_book}.
For all (not necessarily general) smooth Fano 3-folds there is much recent and ongoing work, e.g.\ \cite{liu_rank2, cheltsov_denisova_fujita, cheltsov_park, cheltsov_fujita_kishimoto_okada, denisova, belousov_loginov}.
Other smooth Fano varieties are studied in \cite{zhuang, abban_zhuang}.

The Calabi problem for toric Fano manifolds has been translated into an easy combinatorial condition via differential geometry methods: a toric Fano manifold admits a Kähler--Einstein metric if and only if the moment polytope of its toric boundary has barycentre at the origin \cite{Mabuchi, BatyrevSelivanova}. Blum and Jonsson~\cite{BlumJonsson} found a completely algebraic proof of the fact that the same condition is equivalent to the K-polystability of a (possibly singular) toric Fano variety.
K-stability of Fano varieties with the action of certain algebraic groups (e.g.\ T-varieties or spherical varieties) has been investigated in \cite{IltenSuess, Suess, Delcroix}.

It is natural to investigate the geometry of K-moduli spaces of Fano varieties.
Only few cases are known, e.g.\ smooth del Pezzo surfaces  \cite{mabuchi_mukai,odaka_spotti_sun}, cubic $3$-folds \cite{liu_xu}, quartic $3$-folds \cite{quarticthreefolds} and cubic $4$-folds \cite{liu_cubic_fourfolds}.

Here we present results about a specific family of smooth Fano $3$-folds by using toric geometry and deformation theory techniques which originate in mirror symmetry \cite{ccggsk}.

\subsection{The family \textnumero4.2}
There is only one deformation family of \emph{smooth} Fano $3$-folds with Picard rank $4$ and anticanonical volume $28$: this is the $2$nd entry in the Mori--Mukai list \cite{MM82,MM03} (and also in \cite{iskovskikh_prokhorov}) of families of smooth Fano $3$-folds with Picard rank $4$, thus it is denoted by \textnumero4.2 in \cite{calabi_problem_book}.
This family is denoted by $\mathrm{MM}_{4-3}$ in \cite[\S87]{quantum_periods} because the authors have reordered the families of smooth Fano $3$-folds of Picard rank $4$ so that the anticanonical volume increases along the list.
The other invariants of this family are $h^{1,2}(X) = 1$ and $\chi(T_X) = -1$.

Each member of the family \textnumero4.2 is the blow-up of the cone over a smooth quadric surface $S \subset \PP^3$ with centre the disjoint union of the vertex and an elliptic curve on $S$.
Each member of the family \textnumero4.2 is K-polystable by \cite[\S4.6]{calabi_problem_book}, hence gives a point in $\Kps{3}{28}$, which is the K-moduli space of $3$-dimensional K-polystable Fano varieties with anticanonical volume $28$.

We prove the following:

\begin{theorem}\label{thm:main}
	Let $M = \Kps{3}{28}$ be the K-moduli space of K-polystable Fano $3$-folds with anticanonical volume $28$.
	Let $\modulismooth \subset M$ be the locus which parametrises the smooth K-polystable Fano $3$-folds with anticanonical volume $28$ and Picard rank $4$.
	
	Then there exists an open subscheme $U$ of $M$ such that $U$ is a smooth rational surface, $U \cap \modulismooth \neq \emptyset$ and $U \setminus \modulismooth$ is a smooth rational curve.
\end{theorem}

We therefore conclude that the irreducible component of $\Kps{3}{28}$ containing $\modulismooth$ is a rational surface. 

\subsection{Idea of the proof}

We pick a singular toric Fano $3$-fold $X$ whose singular locus consists of $4$ ordinary double points: this has reflexive ID 735 in the Graded Ring Data Base \cite{GRDB-toric3}. This variety $X$ is K-polystable because the barycentre of the moment polytope of its toric boundary is the origin. It has been shown to deform to the members of the family \textnumero4.2 by Galkin in his PhD thesis \cite{Galkin}.
Moreover, by \cite[\S5.1]{ask_petracci} the base space of the miniversal deformation of $X$ is a polydisc of dimension $4$. 
However, no information about the local structure, in the Zariski topology, of the K-moduli space can be derived from this infinitesimal study.

Here, we construct an explicit algebraic model of the miniversal deformation of $X$ over a Zariski neighbourhood of the origin in the affine space $\AA^4$, i.e.\ we exhibit a flat algebraic deformation of $X$ over $\AA^4$ such that its  formal completion at the origin is the miniversal deformation of $X$.

To achieve this we use the Laurent inversion method introduced in \cite{laurent_inversion} and used to systematically construct (not necessarily smooth) Fano 3-folds in \cite{from_cracked_to_fano,Heuberger_QFLI}.
Indeed, we find a toric $4$-fold $F$ such that $X$ is a divisor in $F$ and we are able to smooth $X$ by deforming it inside the linear system $\vert \cO_F(X) \vert$. We then show that this algebraic deformation induces the miniversal one.

We deduce that the K-moduli space is unirational near the point $[X]$ corresponding to $X$.
Moreover, by using deformation-theoretic techniques in Proposition~\ref{prop:definition_P_X}(8)  we show that the K-moduli space is smooth of dimension $2$ at $[X]$.
This proves the rationality in the statement.

\begin{remark}
%
%
In this paper we describe geometric properties of a $2$-dimensional component of K-moduli of smoothable Fano $3$-folds, by analysing the deformation theory of a singular toric Fano $3$-fold.
In later work, Cheltsov, Duarte Guerreiro, Fujita, Krylov and Martinez-Garcia~\cite[Corollary~1.14]{cheltsov2023kstability} express this rational surface as a GIT quotient. Their methods are independent of ours, and rely on the birational description of smooth members of the family~\textnumero4.2.

\end{remark}

\begin{remark}
	Let $X'$ be the Gorenstein toric Fano $3$-fold which appears in \cite{GRDB-toric3} with reflexive ID~1518 and canonical ID~61936.
	We expect that $X'$ deforms to two different families of smooth Fano $3$-folds, namely \textnumero4.2 (which is the family studied in this paper) and \textnumero2.21 (which is the family consisting of the blowups of the smooth quadric $3$-fold at a twisted quartic curve). Both families have anticanonical volume $28$, but they have different Betti numbers. The general member of the family \textnumero2.21 is K-polystable by \cite{calabi_problem_book}.

	By \cite{cheltsov2023kstability} the two irreducible components of $\Kps{3}{28}$ generically parametrising K-polystable members of the families \textnumero2.21 and \textnumero4.2 belong to distinct connected components of $\Kps{3}{28}$. This is not a contradiction with the expected properties of $X'$ described above because $X'$ is not K-semistable, hence it does not give a point in K-moduli.
\end{remark}

\subsection*{Notation and conventions} \label{sec:notation}
The set of non-negative (resp.\ positive) integers is denoted by $\NN$ (resp.\ $\NN^+$).
We work over an algebraically closed field of characteristic zero, denoted by $\CC$.
Every toric variety or toric singularity is assumed to be normal.
A Fano variety is a normal projective variety whose anticanonical divisor is $\QQ$-Cartier and ample.
A del Pezzo surface is a Fano variety of dimension $2$.


\subsection*{Acknowledgements}
LH would like to thank Tom Coates and Al Kasprzyk for many conversations during the preparation of \cite{Heuberger_QFLI}, which has proved to be a comprehensive testing ground for $3$-fold Laurent inversion. AP wishes to thank Ivan Cheltsov and Anne-Sophie Kaloghiros for many fruitful conversations.

LH is supported by Leverhulme grant RPG-2021-149.
AP acknowledges partial financial support from
INdAM GNSAGA ``Gruppo Nazionale per le Strutture Algebriche, Geometriche e le loro Applicazioni''
and from
PRIN2020 2020KKWT53 ``Curves, Ricci flat Varieties and their Interactions''. 

We use the Magma computational algebra software \cite{Magma} in Section~\ref{LIsec}.

\section{A toric variety}

Here we study a specific toric Fano $3$-fold, which was already studied in \cite[\S5.1]{ask_petracci}. Its associated polytope first appeared in the classification of reflexive polytopes \cite{kreuzer_skarke} by Kreuzer and Skarke and is part of the Graded Ring Database ``Toric canonical Fano $3$-folds" list \cite{GRDB-toric3,Kasprzyk}, appearing with canonical ID $674679$ and reflexive ID $735$.

\begin{prop} \label{prop:definition_P_X}
	In the lattice $N = \ZZ^3$ consider the polytope $P$ whose vertices are:
	\begin{equation} \label{eq:rays_of_X}
		\begin{gathered}
			\rho_1 = \begin{pmatrix}
				1 \\ 0 \\ 0 
			\end{pmatrix}\!, \
			\rho_2 = \begin{pmatrix}
				0 \\ 1 \\ 0 
			\end{pmatrix}\!, \
			\rho_3 = \begin{pmatrix}
				0 \\ 0 \\ 1
			\end{pmatrix}\!, \
			\rho_4 = \begin{pmatrix}
				1 \\ 1 \\ 0
			\end{pmatrix}\!, \
			\rho_5 = \begin{pmatrix}
				1 \\ 0 \\ 1
			\end{pmatrix}\!, \\
			\rho_6 = - \begin{pmatrix}
				1 \\ 0 \\ 0 
			\end{pmatrix}\!, \
			\rho_7 = - \begin{pmatrix}
				0 \\ 1 \\ 0 
			\end{pmatrix}\!, \
			\rho_8 = - \begin{pmatrix}
				0 \\ 0 \\ 1
			\end{pmatrix}\!, \
			\rho_9 = - \begin{pmatrix}
				1 \\ 1 \\ 0
			\end{pmatrix}\!, \
			\rho_{10} = - \begin{pmatrix}
				1 \\ 0 \\ 1
			\end{pmatrix}\!.
		\end{gathered}
	\end{equation}
Let $X$ be the toric $3$-fold associated to the face fan of $P$. Then the following assertions hold.
		\begin{figure}
		\centering

\begin{tikzpicture}[scale=0.8,every node/.style={scale=0.9}]
\draw[gray!80, thick] (1,1) -- (-2,1);
\draw[gray!80, thick] (1,1) -- (-1,3);

\draw[gray!80, thick] (1,1) -- (3,0);
\draw[gray!80, thick] (-1,3) -- (-2,1);
\draw[gray!80, thick] (-3,0) -- (-2,1);
\draw[gray!80, thick] (1,1) -- (-1,-2);
\draw[gray!80, thick] (-1,-2) -- (-2,1);

\draw[black, very thick] (-3,0) -- (-1,3);
\draw[black, very thick] (-3,0) -- (-1,-1);
\draw[black, very thick] (-1,-1) -- (1,2);
\draw[black, very thick] (1,2) -- (3,0);
\draw[black, very thick] (1,2) -- (2,-1);
\draw[black, very thick] (-1,-1) -- (2,-1);
\draw[black, very thick] (2,-1) -- (3,0);
\draw[black, very thick] (1,2) -- (-1,3);
\draw[black, very thick] (3,0) -- (1,-3);
\draw[black, very thick] (2,-1) -- (1,-3);
\draw[black, very thick] (-1,-2) -- (1,-3);
\draw[black, very thick] (-1,-2) -- (-3,0);
\draw[black, very thick] (-1,-1) -- (1,-3);

\draw[gray!80, thick] (0.1,0.1) -- (-0.1,-0.1);
\draw[gray!80, thick] (-0.1,0.1) -- (0.1,-0.1);

\filldraw[gray] (0,0) circle (0pt) node[anchor=north] {\textcolor{black}{$0$}};
\filldraw[black] (2,-1) circle (1pt) node[anchor=south east] {$\rho_1$};
\filldraw[black] (3,0) circle (1pt) node[anchor=north west] {$\rho_4$};
\filldraw[black] (1,1) circle (1pt) node[anchor=north] {$\rho_2$};
\filldraw[black] (-2,1) circle (1pt) node[anchor=north west] {$\rho_6$};
\filldraw[black] (-3,0) circle (1pt) node[anchor=east] {$\rho_9$};
\filldraw[black] (1,-3) circle (1pt) node[anchor=north] {$\rho_8$};
\filldraw[black] (-1,-2) circle (1pt) node[anchor=north east] {$\rho_{10}$};
\filldraw[black] (1,2) circle (1pt) node[anchor=south west] {$\rho_5$};
\filldraw[black] (-1,3) circle (1pt) node[anchor=south] {$\rho_3$};
\filldraw[black] (-1,-1) circle (1pt) node[anchor=north ] {$\rho_7$};

%
\end{tikzpicture}
		\caption{The polytope $P$ in Proposition~\ref{prop:definition_P_X}}
		\label{fig:P735}
	\end{figure}
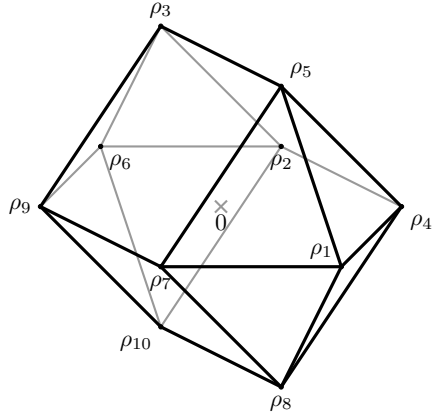
	\begin{enumerate}
		\item $X$ is a K-polystable Fano $3$-fold with anticanonical volume $28$.
		\item The singular locus of $X$ consists of $4$ ordinary double points.
		\item The base of the  miniversal deformation of $X$ is smooth of dimension $4$, i.e.\ $\Def{X}$ is isomorphic to the formal spectrum of the power series ring $\CC \pow{t_\al, t_\beta, t_\gamma, t_\delta}$.
		\item $X$ deforms to the $2$nd entry in the Mori--Mukai list \cite{MM82,MM03} of smooth Fano $3$-folds with Picard rank $4$ (the other invariants are $(-K)^3 = 28$ and $h^{1,2} = 1$).
		\item The automorphism group $\Aut(P) \subset \mathrm{GL}(N) \simeq \mathrm{GL}_3(\ZZ)$ of the polytope $P$ is isomorphic to the semidirect product $D_8 \rtimes C_2$, where $C_2$ is the cyclic group of order $2$ and $D_8$ is the dihedral group of order $8$; more precisely, $\Aut(P)$ is generated by the following matrices:
		\begin{align*}
			g_1 &= \begin{pmatrix}
			-1 & 1 & 1 \\ 0 & 1 & 0 \\ 0 & 0 & 1
		\end{pmatrix}
		\text{ of order } 2, \\
			g_2 &= \begin{pmatrix}
			1 & 0 & 0 \\ 0 & 0 & 1 \\ 0 & 1 & 0
			\end{pmatrix}
		\text{ of order } 2, \\
			g_3 &= \begin{pmatrix}
			1 & -1 & 0 \\ 0 & 0 & 1 \\ 0 & -1 & 0
		\end{pmatrix}
		\text{ of order } 4.
		\end{align*}
		\item The automorphism group $\Aut(X)$ of $X$ is isomorphic to $T_N \rtimes \Aut(P)$, where $T_N = N \otimes_\ZZ \CC^*$ is the $3$-dimensional algebraic torus whose cocharacter lattice is $N$.
		\item The formal action of $\Aut(X)$ on $\Def{X} \simeq \Spf \CC \pow{t_\al, t_\beta, t_\gamma, t_\delta}$ is given by:
		\begin{enumerate}[label=$\bullet$]
			\item the torus $T_N$ acts linearly and diagonally on $t_\al, t_\beta, t_\gamma, t_\delta$ with weights $(0,1,1)$, $(0,1,-1)$, $(0,-1,-1)$ $(0,-1,1)$ in $M$, respectively,
			\item the involution $g_2$ leaves $t_\al$ and $t_\gamma$ fixed and swaps $t_\beta$ and $t_\delta$,
			\item the involution $g_1$ fixes everything,
			\item $g_3$ acts as $t_\al \mapsto - t_\beta$, $t_\beta \mapsto t_\gamma$, $t_\gamma \mapsto t_\delta$, $t_\delta \mapsto - t_\al$.
		\end{enumerate}
		\item The scheme $\Kps{3}{28}$ is smooth of dimension $2$ in a neighbourhood of the point $[X]$ corresponding to $X$. More precisely, $t_\al t_\beta t_\gamma t_\delta$ and $t_\al t_\gamma - t_\beta t_\delta$ are regular formal parameters of $\Kps{3}{28}$ at $[X]$.
		\item In a neighbourhood of $[X]$ in $\Kps{3}{28}$ the locus of non-smooth K-polystable Fano varieties is a smooth curve passing through $[X]$.
	\end{enumerate}
\end{prop}

The proof of (1--4) is contained in  \cite[\S5.1]{ask_petracci}; here we recap the argument briefly.

\begin{proof}
	The polytope $P$ has $20$ edges  (all of them have lattice length $1$) and has $12$ $2$-dimensional faces ($8$ of them are triangles and $4$ of them are quadrilaterals). It is depicted in Figure~\ref{fig:P735}.
	Let $\Sigma_X$ be the face fan of $P$.
	The $3$-dimensional cones of $\Sigma_X$ are
	\begin{equation}
		\begin{gathered}
			\sigma_{145}, \ \sigma_{157}, \ \sigma_{148}, \ \sigma_{178}, \\
			\sigma_{6910}, \ \sigma_{2610}, \ \sigma_{369}, \ \sigma_{236}, \\
			\sigma_{2345}, \ \sigma_{24810}, \
			\sigma_{78910}, \  \sigma_{3579}.
		\end{gathered}
	\end{equation}
Here $\sigma_{ijk}$ (resp.\ $\sigma_{ijkl}$) is the $3$-dimensional cone in $N$ whose rays are $\rho_i$, $\rho_j$, $\rho_k$ (resp.\ $\rho_i$, $\rho_j$, $\rho_k$, $\rho_l$).
	$X$ is the toric variety associated to $\Sigma_X$.
	
	\medskip
	
	(1) $X$ is Fano, because it is defined by the face fan of a Fano polytope.
Let $P^\circ$ be the polar of $P$: it is the moment polytope of the toric boundary of $X$.
	Since $P^\circ$ has normalised volume $28$, we have $(-K_X)^3 = 28$.
	Since $P = -P$ we have that the barycentre of $P^\circ$ is the origin, therefore $X$ is K-polystable by \cite[Corollary 1.2]{berman_polystability} and \cite[Corollary~7.17]{BlumJonsson}.
	
	\medskip
	
	(2)	The simplicial cones of $\Sigma_X$ are smooth.
	The non-simplicial cones of $\Sigma_X$ are $\sigma_\al = \sigma_{2345}$, $\sigma_\beta = \sigma_{24810}$, $\sigma_\gamma = \sigma_{78910}$, $\sigma_\delta = \sigma_{3579}$.
	Let $U_\al$ (resp.\ $U_\beta$, $U_\gamma$, $U_\delta$) be the toric open affine subscheme of $X$  associated to the cone $\sigma_\al$ (resp.\ $\sigma_\beta$, $\sigma_\gamma$, $\sigma_\delta$), i.e.\ $U_\al = \Spec \CC[\sigma^\vee \cap M]$.
	
	Each of the cones $\sigma_\al$, $\sigma_\beta$, $\sigma_\gamma$, $\sigma_\delta$ is $\GL_3(\ZZ)$-equivalent to the cone over a unit square placed at height $1$, i.e.\ to the cone spanned by $e_3, e_1+e_3$, $e_2 + e_3$, $e_1 + e_2 + e_3$ in $\ZZ^3$ where $e_1, e_2, e_3$ is the standard basis of $\ZZ^3$.
	Therefore $U_\al$, $U_\beta$, $U_\gamma$, $U_\delta$ are all isomorphic to $\Spec \CC[x,y,z,w] / (xy-zw)$, which has an ordinary double point. Hence the singular locus of $X$ is made up of $4$ ordinary double points.
	
	More precisely, we fix an isomorphism between $U_\al$ (resp.\ $U_\beta$, $U_\gamma$, $U_\delta$)  and $\Spec \CC[x,y,z,w] / (xy-zw)$ by giving names to the minimal generators of the monoid $\sigma_\al^\vee \cap M$ (resp.\ $\sigma_\beta^\vee \cap M$, $\sigma_\gamma^\vee \cap M$, $\sigma_\delta^\vee \cap M$):
	\begin{equation*}
		\begin{gathered}
					x_\al = (-1,1,1) \quad y_\al = (1,0,0) \quad z_\al = (0,0,1) \quad w_\al = (0,1,0), \\
			x_\beta = (0,0,-1) \quad y_\beta = (0,1,0) \quad z_\beta = (-1,1,0) \quad w_\beta = (1,0,-1), \\
			x_\gamma = (0,-1,0) \quad y_\gamma = (0,0,-1) \quad z_\gamma = (-1,0,0) \quad w_\gamma = (1,-1,-1), \\
			x_\delta = (0,0,1) \quad y_\delta = (0,-1,0) \quad z_\delta = (-1,0,1) \quad w_\delta = (1,-1,0).
		\end{gathered}
	\end{equation*}

\medskip
	
	(3) Since $\rH^1(T_X) = 0$ and $\rH^2(T_X) = 0$ \cite[Lemma 4.4]{petracci_survey} and $X$ has isolated singularities, the deformations of $X$ can be identified with the deformations of the singularities of $X$ (see for instance \cite[\S14.2.1]{petracci_survey}).
	More precisely,  the product of the restriction maps
	\begin{equation} \label{eq:restriction_map_def_functors}
		\Def{X} \longrightarrow \Def{U_\al} \times \Def{U_\beta} \times \Def{U_\gamma} \times \Def{U_\delta}
	\end{equation}
 is smooth and induces an isomorphism on tangent spaces.
	Since an ordinary double point has a smooth $1$-dimensional miniversal deformation space, we conclude.
	
For future use we need to specify the meaning of $t_\al$, $t_\beta$, $t_\gamma$, $t_\delta$. We will always use the smooth functor
\begin{equation*}
	\Spf \CC \pow{t_\al} \longrightarrow \Def{U_\al},
\end{equation*}
	which induces an isomorphism on tangent spaces and is induced by the formal deformation of $U_\al$ over $\Spf \CC \pow{t_\al}$ given by the equation
\begin{equation} \label{eq:meaning_of_t_alpha}
	x_\al y_\al - z_\al w_\al + t_\al = 0,
\end{equation}
where the meaning of $x_\al$, $y_\al$, $z_\al$, $w_\al$ is specified in the proof of (2) above.
We make the similar conventions for $t_\beta$, $t_\gamma$ and $t_\delta$ and we will always stick to these conventions.
	
	\medskip
	
	(4) By \cite[Proposition~2.5]{ask_petracci} the $2$nd Betti number of $X$ is $4$. Since the smoothing of a $3$-fold ordinary double point has a Milnor fibre homotopically equivalent to the $3$-sphere, the $2$nd Betti number of the general smoothing of $X$ is $4$.
	Moreover, in smoothing $X$, also the anticanonical degree is preserved.
	In the Mori--Mukai list \cite{MM82, MM03} there is only one family of smooth Fano $3$-folds with Picard rank $4$ and anticanonical degree $28$.
	
	\medskip
	
	(5) We now compute the automorphism group of the polytope $P$. Any automorphism of $P$ should preserve the $\rho_1$-$\rho_6$ axis. There is only one reflection that switches $\rho_1$ and $\rho_6$, and we choose it as the generator of a $C_2$-subgroup of $\Aut(P)$. As a matrix in $\GL(N) = \GL_3(\ZZ)$ the generator of the $C_2$ above is 
	\[g_1=\begin{pmatrix}
	-1 & 1& 1 \\
 	0 & 1 & 0 \\
	0 & 0 & 1
	\end{pmatrix}.\]
All other automorphisms of $P$ fix $\rho_1$. They must also preserve the square obtained by intersecting $H$ with $P$, where $H = \RR \frac{\rho_2 + \rho_4}{2} + \RR \frac{\rho_3 + \rho_5}{2}$. This square has symmetry group $D_8$ and all of its symmetries lift uniquely to symmetries of $P$ that fix $\rho_1$. Moreover, these do not commute with the generator of $C_2$, thus $\Aut(P)= D_8 \rtimes C_2$. The matrix representations of the two generators of $D_8$ are:
	\[g_2=\begin{pmatrix}
	1 & 0& 0 \\
	0 & 0 & 1\\
	0 & 1 &0
	\end{pmatrix} \quad \text{ and } \quad
	g_3=\begin{pmatrix}
	1 & -1 & 0 \\
 	0 & 0 & 1 \\
	0 & -1 & 0
	\end{pmatrix}.
	\] 
	
	\medskip
	
	(6) This follows from \cite[Proposition~2.8]{ask_petracci} because no facet of the polar polytope $P^\circ$ has interior lattice points.
	
	\medskip
	
	(7) 	The vector space $\TT^1_X$, which is the tangent space of $\Def{X}$, is a $4$-dimensional representation of $\Aut(X) = T_N \rtimes \Aut(P)$.
	Using \eqref{eq:restriction_map_def_functors} we identify $\TT^1_X$ with the tangent space to the deformations of the $4$ singular points of $X$.
	
	The vector $(0,1,1) \in M$ is the only one such that by pairing with the primitive generators of the cone $\sigma_\al$, namely $\rho_2$, $\rho_3$, $\rho_4$, $\rho_5$, we obtain $1$. In other words, $-(0,1,1) \in M$ is the vertex of $P^\circ$ which is dual to the face of $P$ whose vertices are $\rho_2$, $\rho_3$, $\rho_4$, $\rho_5$.
	By \cite{altmann_t1} $\TT^1_{U_\al}$ is the $1$-dimensional representation of the torus $T_N = N \otimes_\ZZ \CC^*$ with weight $(0,-1,-1)$. This implies that the coordinate $t_\al$ has degree $(0,-1,-1) \in M$ with respect to $T_N$.
	
	In similar ways one proves that $t_\beta$, $t_\gamma$, $t_\delta$ have degrees $(0,-1,1)$, $(0,1,1)$, $(0,1,-1)$ respectively. We now focus on the action of the $3$ generators $g_1$, $g_2$, $g_3$ of $\pi_0(\Aut(X))$ given in (5).
	
	\smallskip
	
	The involution $g_1$ acts on the vertices of $P$ by $\rho_1 \leftrightarrow \rho_6$,  
	$\rho_2 \leftrightarrow \rho_4$,
		$\rho_3 \leftrightarrow \rho_5$,
		$\rho_7 \leftrightarrow \rho_9$,
			$\rho_8 \leftrightarrow \rho_{10}$.
	In particular, by remembering that the torus orbits on $X$ correspond bijectively to the cones of the face fan of $P$, we can see that $U_\al$ is $g_1$-invariant and $g_1$ acts via pull-back on the regular functions of $U_\al$. For example
	\[
	g_1(x_\al) = \begin{pmatrix}
		-1 & 1 & 1
	\end{pmatrix}
\begin{pmatrix}
	-1 & 1 & 1 \\
	0 & 0 & 1 \\
	0 & 1 & 0
\end{pmatrix}
=
\begin{pmatrix}
	1 & 0 & 0
\end{pmatrix}
= y_\al
	\]
	and
		\[
	g_1(z_\al) = \begin{pmatrix}
		0 & 0 & 1
	\end{pmatrix}
	\begin{pmatrix}
		-1 & 1 & 1 \\
		0 & 0 & 1 \\
		0 & 1 & 0
	\end{pmatrix}
	=
	\begin{pmatrix}
		0 & 1 & 0
	\end{pmatrix}
	= w_\al.
	\]
	This shows that $g_1$ maps the miniversal deformation \eqref{eq:meaning_of_t_alpha} of $U_\al$ to itself. Hence $g_1$ maps $t_\al$ to itself.
	In a similar way, one can show that $U_\beta$, $U_\gamma$ and $U_\delta$ are $g_1$-invariant and $g_1$ fixes $t_\beta$, $t_\gamma$ and $t_\delta$ too.
	
	\smallskip
	
	Let us now focus on $g_2$. On the vertices of $P$ this involution acts as $\rho_2 \leftrightarrow \rho_3$, $\rho_4 \leftrightarrow \rho_5$, $\rho_7 \leftrightarrow \rho_8$, $\rho_9 \leftrightarrow \rho_{10}$ and fixes $\rho_1$ and $\rho_6$.
	One can see that $U_\al$ is $g_2$-invariant and $g_2(x_\al) = x_\al$, $g_2(y_\al) = y_\al$, $g_2(z_\al) = w_\al$. Therefore $g_2$ maps the miniversal deformation \eqref{eq:meaning_of_t_alpha} of $U_\al$ to itself, hence $t_\al$ is fixed by $g_2$.
	
	One sees that $g_2(U_\beta) = U_\delta$. So, via pull-back, $g_2$ maps a regular function on $U_\delta$ to a regular function on $U_\beta$.
	A similar calculation as above gives $g_2(y_\delta)=x_\beta$, $g_2(x_\delta) = y_\beta$, $g_2(z_\delta) = z_\beta$, $g_2(w_\delta) = w_\beta$.
	So the miniversal deformation $x_\delta y_\delta - z_\delta w_\delta + t_\delta = 0$ of $U_\delta$ is mapped by $g_2$ to 
	the deformation $x_\beta y_\beta - z_\beta w_\beta + t_\delta = 0$ of $U_\beta$. This latter deformation is induced by the chosen miniversal deformation $x_\beta y_\beta - z_\beta w_\beta + t_\beta = 0$ of $U_\beta$ by setting $t_\delta = t_\beta$. Therefore $g_2$ swaps $t_\delta$ and $t_\beta$.
	
	In a similar way, one can show that $g_2$ leaves $t_\gamma$ fixed.

	\smallskip
	
	Let us now focus on $g_3$.
	The order $4$ matrix $g_3$ acts as follows on the vertices of $P$: $\rho_1$ and $\rho_6$ are fixed and there are two orbits of cardinality $4$, namely $\rho_2 \mapsto \rho_{10} \mapsto \rho_9 \mapsto \rho_3 \mapsto \rho_2$ and $\rho_5 \mapsto \rho_4 \mapsto \rho_8 \mapsto \rho_7 \mapsto \rho_5$.
	
	One sees $g_3(U_\al) = U_\beta$.
	Via pull-back along $g_3$, we get $x_\beta \mapsto w_\al$, $y_\beta \mapsto z_\al$, $z_\beta \mapsto x_\al$, $w_\beta \mapsto y_\al$.
	So the miniversal deformation $x_\beta y_\beta - z_\beta w_\beta + t_\beta = 0$ of $U_\beta$ is mapped via $g_3$ to the deformation $w_\al z_\al - x_\al y_\al + t_\beta = 0$ of $U_\al$.
	This shows that $g_3$ maps $t_\al$ to $- t_\beta$.
	
	In a similar way, one obtains $g_3(U_\beta) = U_\gamma$, $g_3(U_\gamma) = U_\delta$, $g_3(U_\delta) = U_\al$ and that $g_3$ acts as $t_\beta \mapsto t_\gamma \mapsto t_\delta \mapsto - t_\al$.

	\medskip
	
	(8) By (7) the invariant subring of the formal action of the torus $T_N \simeq (\CC^*)^3$ on $\CC \pow{t_\al, t_\beta, t_\gamma, t_\delta}$ is  $\CC \pow{t_\al t_\gamma, t_\beta t_\delta}$, which is a power series ring in $2$ indeterminates.
	The involutions $g_1$ and $g_2$ act trivially on $\CC \pow{t_\al t_\gamma, t_\beta t_\delta}$, whereas $g_3$ acts on $\CC \pow{t_\al t_\gamma, t_\beta t_\delta}$ with order $2$ and maps $t_\al t_\gamma$ to $- t_\beta t_\delta$.
	In conclusion, the subring of $\CC \pow{t_\al, t_\beta, t_\gamma, t_\delta}$ of the invariants elements under the action of $\Aut(X)$ is $\CC \pow{t_\al t_\beta t_\gamma t_\delta, t_\al t_\gamma - t_\beta t_\delta}$, which is a power series ring in $2$ indeterminates.
	
	By the Luna  \'etale slice theorem for algebraic stacks \cite{luna_etale_slice_stacks} the local structure of $\Kss{3}{28} \to \Kps{3}{28}$ at the point $[X]$ is given by the  commutative square
	\begin{equation*}
		\xymatrix{
			\left[ \Spf \CC \pow{t_\al, t_\beta, t_\gamma, t_\delta} \ / \ \Aut(X)  \right] \ar[d] \ar[r] & \Kss{3}{28}  \ar[d] \\
			\Spf \CC \pow{t_\al t_\beta t_\gamma t_\delta, t_\al t_\gamma - t_\beta t_\delta} \ar[r] &  \Kps{3}{28}
		}
	\end{equation*}
	where the horizontal maps are formally \'etale and maps the closed point to $[X]$.
	This implies that $\Kps{3}{28}$ is smooth of dimension $2$ near $[X]$ and that $t_\al t_\beta t_\gamma t_\delta$, $t_\al t_\gamma - t_\beta t_\delta$ are local analytic coordinates.

	\medskip
	
	(9) Here, to simplify, we work in the analytic category.
	Since $t_\al, t_\beta, t_\gamma, t_\delta$ are the smoothing parameters of an ordinary double point, it is clear that in the miniversal deformation of $X$, which now we think over a small polydisc $\Delta^4$ of dimension $4$ (with local analytic coordinates $t_\al, t_\beta, t_\gamma, t_\delta$), the smooth fibres are exactly those where $t_\al \neq 0$, $t_\beta \neq 0$, $t_\gamma \neq 0$, $t_\delta \neq 0$.
	Let $D$ be the discriminant locus, i.e.\ the locus in the base space $\Delta^4$ of the miniversal deformation of $X$ where the fibres are singular; hence $D = \{ t_\al t_\beta t_\gamma t_\delta = 0  \} \subset \Delta^4$ is the union of $4$ hyperplanes.
	
	Let $\Delta^2$ be a small $2$-dimensional polydisc which gives the local structure of $\Kps{3}{28}$ at $[X]$.
	The holomorphic map $\Delta^4 \to \Delta^2$ is given by
	\[
	(t_\al, t_\beta, t_\gamma, t_\delta) \mapsto (t_\al t_\beta t_\gamma t_\delta, t_\al t_\gamma - t_\beta t_\delta).
	\]
	The image of $D$ via this map is the locus in $\Delta^2$ where the first coordinate is zero, therefore it is a germ of a smooth curve passing through the origin.
	The image of $D$ is exactly the locus in $\Delta^2$ of singular K-polystable Fanos close to $[X]$.
\end{proof}

An immediate consequence of Proposition~\ref{prop:definition_P_X} is:

\begin{cor} \label{cor:almost_everything_without_rationality}
	Let $M = \Kps{3}{28}$ be the K-moduli space of K-polystable Fano $3$-folds with anticanonical volume $28$.
	Let $\modulismooth \subset M$ be the locus which parametrises the smooth K-polystable Fano $3$-folds with anticanonical volume $28$ and Picard rank $4$.
	Then there exists an open subscheme $U$ of $M$ such that $U$ is a smooth surface and $U \setminus \modulismooth$ is a smooth curve.
\end{cor}

This is a portion of Theorem~\ref{thm:main}; the only things which are missing are the statements about the rationality of $U$ and of $U \setminus \modulismooth$.

\section{Embeddings in toric varieties}

Now we realise the toric Fano $3$-fold $X$ of Proposition~\ref{prop:definition_P_X} inside a toric Fano $4$-fold.

\begin{prop} \label{prop:X_inside_F}
	In the lattice $N_F = \ZZ^4$ consider the lattice polytope $P_F$ whose vertices are the columns $r_1, \dots, r_7 \in N_F$ of the matrix
	\[
	R_F=
	\begin{bmatrix}
		1 & 0 & -1 & 0 & 0 & 0 & 0 \\
		0 & 1 & 0 & -1 & 0 & 0 & 0 \\
		0 & 0 & -1 & -1 & 1 & -1 & 0 \\
		0 & 0 & -1 & -1 & 0 & -1 & 1
	\end{bmatrix}.
	\]
	Let $F$ be the toric variety associated to the face fan of $P_F$.
	Then:
	\begin{enumerate}
		\item $F$ is a smooth Fano $4$-fold and is the GIT quotient $\AA^7 \git (\CC^*)^3$ given by the weight matrix 
		\[\begin{matrix}
			u_1& u_2& u_3& u_4&u_5&u_6&u_7\\
			\hline 
1 & 0 & 1 & 0 & 0 & -1 & 0 \\
0 & 1 & 0 & 1 & 0 & -1 & 0 \\
0 & 0 & 0 & 0 & 1 & 1 & 1
		\end{matrix}\]
	and by the stability condition $(1,1,1)$;
	\item if $X$ is the toric Fano $3$-fold considered in Proposition~\ref{prop:definition_P_X}, then the linear map $N_X = \ZZ^3 \to N_F = \ZZ^4$ induced by the matrix
	\[
	A = \begin{bmatrix}
		0 & 1 & 0 \\
		0 & 0 & 1 \\
		-1 & 1 & 1 \\
		1 & 0 & 0
	\end{bmatrix}
	\]
induces a toric morphism $X \to F$ which is a closed embedding and identifies $X$ with the divisor in $F$ defined by the equation
\begin{equation*}
u_5 u_7 - u_6^2 u_1 u_2 u_3 u_4 = 0
\end{equation*}
in the Cox coordinates of $F$.
	\end{enumerate}
\end{prop}	
\begin{remark}
The matrix $A$ gives an injective homomorphism of algebraic tori $(\CC^*)^3 \into (\CC^*)^4$ and $X$ is the closure in $F$ of the orbit of $(1,1,1,1)$ under the action of $(\CC^*)^3$ induced by $A$.
\end{remark}

\begin{remark}
The toric variety $F$ in Proposition~\ref{prop:X_inside_F} is a smooth Fano $4$-fold with ID 97 in \cite{GRDB-toricsmooth}.
\end{remark}

\begin{remark}
The weight matrix of $F$ and the line bundle
\[
L=~\mathcal{O}_F\begin{pmatrix}0\\0\\2\end{pmatrix}
\]
 of which the equation of $X$ is a section of have previously appeared in \cite[\S 87]{quantum_periods}. We refer to the analysis there justifying why the Picard rank of $X$ is $4$, whereas the Picard rank of its ambient space $F$ is $3$.
\end{remark}

\begin{proof}[Proof of Proposition~\ref{prop:X_inside_F}]
(1) The divisor sequence \cite{cox_little_schenck}[Theorem 4.1.3] of $F$ is
\begin{equation*}
	0 \longrightarrow M_F \simeq \ZZ^4 \xrightarrow{(R_F)^T}  \Div_T(F) \simeq \ZZ^7 \overset{D_F}{\longrightarrow}  \Pic(F) \simeq \ZZ^3 \longrightarrow 0
\end{equation*}
where $D_F$ is the $3 \times 7$ matrix in the statement (1).

Thus $D_F$ is the weight matrix for $F$; in other words, the action of the torus $(\CC^*)^3$ on $\AA^7$ with weights given by $D_F$ is such that $F$ is the corresponding GIT quotient $\AA^7 \git (\CC^*)^3$ with respect to a certain stability condition.
Since $F$ is Fano and $-K_F$ is the $(\CC^*)^3$-linearised line bundle on the quotient with weights $(1,1,3) \in \Pic(F) \simeq \ZZ^3$, the stability condition is given by the chamber which contains $(1,1,3)$. This chamber also contains $(1,1,1)$, as it is the positive orthant.

 It is immediate to check that $F$ is smooth as all relevant $3 \times 3$ minors of $R_F$ are equal to $\pm1$.

\medskip

(2) Let $\rho_1, \dots, \rho_{10}$ be the vectors in \eqref{eq:rays_of_X}, i.e.\ the primitive generators of the rays of the fan defining $X$. 
The map induced by $A$ sends the $\rho_i$'s onto the hyperplane with normal vector $m_X:=(1,1,-1,-1)$: 
\[ A 
\left(\begin{array}{c|c|c}
\rho_1 & \cdots & \rho_{10}
\end{array}\right)
= \begin{bmatrix}
		0 & 1 &0 &1 &0 &0 &-1& 0 &-1& 0  \\
		0 &0 &1 &0 &1 &0 &0 &-1& 0& -1\\
		-1 &1& 1& 0& 0& 1& -1& -1& 0& 0\\
		1& 0& 0& 1& 1& -1& 0 &0 &-1 &-1
	\end{bmatrix}.
	\]
One can check that each cone of $\Sigma_X$ is sent via $A$ to a cone in $\Sigma_F$. Furthermore, the image of the fan $\Sigma_X$ equals the intersection of the hyperplane $m_X^\perp$ with $\Sigma_F$.
Therefore the image of the toric morphism $X \to F$ is a prime divisor on $F$. We want to determine its equation in the Cox coordinates of $F$.

Each lattice element $A(\rho_i)$ belongs to exactly one 2-dimensional cone of $F$.
Let us write it as a linear combination, with positive coefficients, of the primitive generators of the fan $\Sigma_F$. For instance
\[
A(\rho_1)=(0,0,-1,1)^T=(0,0,-1,-1)^T+2(0,0,0,1)^T=r_6+2r_7.
\]
The matrix which express these combinations is
\[
\begin{array}{c|ccccccc}
	& r_1 &r_2&r_3&r_4 &r_5& r_6 &r_7  \\
	\hline   A(\rho_1) &  0 &0 &0 &0 &0 &1 &2   \\
	A(\rho_2) & 	1& 0 &0& 0& 1& 0& 0 \\
	A(\rho_3) & 	0 &1& 0 &0 &1 &0 &0  \\
	A(\rho_4) & 	1 &0 &0 &0 &0& 0& 1 \\
	A(\rho_5) &     0 &1 &0& 0& 0& 0& 1 \\
	A(\rho_6) &     0 &0 &0& 0& 2& 1& 0 \\
	A(\rho_7) &     0 &0 &1& 0& 0& 0& 1\\ 
	A(\rho_8) &     0 &0 &0& 1& 0& 0& 1\\ 
	A(\rho_9) &     0 &0 &1& 0& 1& 0& 0\\ 
	A(\rho_{10}) &  0 &0 &0& 1& 1& 0& 0\\ 
\end{array}
\]
Let $B$ be the $7 \times 10$ matrix given by transposing this matrix.
We have a commutative diagram with exact rows
\begin{equation*}
	\begin{aligned}
		\xymatrix{
			&&&&& \ZZ  \\
			0  \ar[r] & (\Pic(F))^\vee \ar[rr] &&(\Div_T F)^\vee  \ar[rr]^{R_F} &&N_F\ar@{->>}[u]^{m_X}\ar[r] &0 \\
			0  \ar[r]& (\Cl(X))^\vee \ar[u] \ar[rr]&& (\Div_T X)^\vee \ar[u]^{B^T}  \ar[rr]^{R_X} &&N_X\ar[u]^{A}\ar[r] &0 \\
		}
	\end{aligned}
\end{equation*}
which dualises to
\begin{equation*}
	\begin{aligned}
		\xymatrix{ & \ZZ \ar[d]^{m_X} \\
			0  \ar[r] & M_F \ar[d]^{A^T} \ar[rr]^{R_F^T} && \Div_T F  \ar[rr]^{D_F} \ar[d]^B &&\Pic(F) \ar[d] \ar[r] &0 \\
			0  \ar[r]& M_X \ar[rr]^{R_X^T} &&  \Div_T X   \ar[rr]^{R_X} && \Cl(X) \ar[r] &0 
		}
	\end{aligned}
\end{equation*}
where the two exact rows are the divisor sequences. The map $B \colon \Div_T F \to \Div_T X$ is the pull-back of torus-invariant divisors on $F$ along the morphism $X \to F$. In terms of Cox coordinates we have the following.
\begin{align*}
u_1& \mapsto x_2x_4\\
u_2& \mapsto x_3x_5\\
u_3& \mapsto x_7x_9\\
u_4& \mapsto x_8x_{10}\\
u_5& \mapsto x_2x_3x_6^2x_9x_{10}\\
u_6& \mapsto x_1x_6\\
u_7& \mapsto x_1^2x_4x_5x_7x_8
\end{align*}
It is clear that the equation describing the image of $X$ in $F$ is
\[
u_5u_7=u_6^2u_1u_2u_3u_4.
\]

It remains to check that the toric morphism $X \to F$ is a closed embedding. This can be checked locally by analysing the affine charts.
\end{proof}

\section{Deforming a toric variety}

In Proposition~\ref{prop:X_inside_F} we have seen that the toric Fano $3$-fold $X$ of Proposition~\ref{prop:definition_P_X} is a divisor inside a smooth toric Fano $4$-fold $F$.
Now we show that a particular $4$-dimensional subspace of the linear system $\vert \cO_F(X) \vert$ gives the miniversal ($\QQ$-Gorenstein) deformation of $X$.

\begin{prop} \label{prop:deforming_X_inside_F}
	Let $X$ be the toric Fano $3$-fold in Proposition~\ref{prop:definition_P_X}
	and let $F$ be the toric Fano $4$-fold in Proposition~\ref{prop:X_inside_F}.
	Let $u_1, \dots, u_7$ denote the Cox coordinates of $F$ as in Proposition~\ref{prop:X_inside_F}.
	Consider the $4$-parameter flat family
	\[
	\cX \to \AA^4 = \Spec \CC [c_1, c_2, c_3, c_4]
	\]
	 given by the equation
	 \begin{equation} \label{eq:equation_of_X_inside_F}
	 		u_5u_7-u_1u_2u_3u_4u_6^2+u_6^2(c_1u_1^2u_2^2+c_2 u_1^2u_4^2+ c_3 u_2^2u_3^2+c_4 u_3^2u_4^2)=0
	 \end{equation}
	inside $F$.
	
	Then the base change of $\cX \to \AA^4$ to $\Spf \CC \pow{c_1, c_2, c_3, c_4}$ is the miniversal ($\QQ$-Gorenstein) deformation of $X$.
	Moreover, the discriminant locus (i.e.\ the locus in $\AA^4$ where the fibres of $\cX \to \AA^4$ are singular) and the divisor $\{ c_1 c_2 c_3 c_4 = 0  \} \subset \AA^4$ coincide in a neighbourhood of the origin in $\AA^4$.
\end{prop}

We put the word $\QQ$-Gorenstein in parenthesis because the miniversal deformation of $X$ coincides with the miniversal $\QQ$-Gorenstein deformation of $X$, since every infinitesimal deformation of $X$ is automatically $\QQ$-Gorenstein because $X$ is Gorenstein.

The fact that the fibre of $\cX \to \AA^4$ over the origin is exactly $X$ is the content of Proposition~\ref{prop:X_inside_F}, hence we need to prove the versality of this deformation.
Before doing this we prove some preliminary lemmata.

\begin{lemma} \label{lemma:sequence_of_polynomials}
Let $s_1, s_2, s_3, x, y$ be indeterminates. Consider the polynomial ring $R = \CC[s_1, s_2, s_3, x, y]$ with $\NN$-grading given by $\deg s_1 = \deg s_2 = \deg s_3 = 1$ and $\deg x = \deg y = 0$.
 For every $k \in \NN$,
 \begin{itemize}
 	\item  let $R_k$ be the homogenous summand of $R$ of degree $k$, i.e.\ the $\CC$-vector subspace of $R$ with basis
 	\[
 	\left\{ s_1^{i_1} s_2^{i_2} s_3^{i_3} x^m y^n \mid i_1, i_2, i_3, m,n \in \NN, \ i_1 + i_2 + i_3 = k  \right\};
 	\]
 	\item set $R_{\leq k} = \bigoplus_{0\leq i \leq k} R_i$, which is the $\CC$-vector subspace of $R$ with basis
 	\[
 	\left\{ s_1^{i_1} s_2^{i_2} s_3^{i_3} x^m y^n \mid i_1, i_2, i_3, m,n \in \NN, \ i_1 + i_2 + i_3 \leq k  \right\};
 	\]
 	\item set $R_{\geq k} = \bigoplus_{ i \geq k} R_i$, which is the ideal $(s_1,s_2,s_3)^k$ of $R$.
 \end{itemize}

Then there exist two sequences $\{  x_k \}_{k \in \NN}$ and $\{  y_k \}_{k \in \NN}$ of polynomials such that:
\begin{enumerate}
	\item $x_0 = x$ and $y_0 = y$,
	\item for every $k \in \NN^+$, $x_k \in R_{\leq k} \cap (x,y)$ and $y_k \in R_{\leq k} \cap (x,y)$,

	\item for every $k \in \NN^+$, $x_k - x_{k-1} \in R_{k}$ and $y_k - y_{k-1} \in R_{k}$,
	
	\item for every $k \in \NN^+$, $f_k := x_k y_k - (xy + s_1 x^2 + s_2 y^2 + s_3 x^2 y^2)$ lies in the ideal $R_{\geq k+1} \cap (x,y)^2$.

\end{enumerate}
\end{lemma}

By (3) the sequences of the $x_k$'s and of the $y_k$'s give two elements of the ring  $\CC[x,y] \pow{s_1, s_2, s_3}$. By (4) the product of these two power series is $xy + s_1 x^2 +s_2 y^2 + s_3 x^2 y^2$.

\begin{proof}[Proof of Lemma~\ref{lemma:sequence_of_polynomials}]
	We proceed by induction on $k \in \NN^+$. In the base case $k=1$ we take $x_1 = x + s_2 y + s_3 x^2 y$ and $y_1 = y + s_1 x$. Obviously (2) and (3) hold. Moreover,
	\[
	f_1 = x_1 y_1 - (xy + s_1 x^2 + s_2 y^2 + s_3 x^2 y^2) = s_1 s_2 xy + s_1 s_3 x^3 y
	\]
	lies in $(s_1, s_2, s_3)^2 \cap (x,y)^2$; this gives (4).
	
	Now we do the inductive step. Fix $k \in \NN^+$ and assume that we have $x_0, \dots, x_k$ and $y_0, \dots, y_k$ which satisfy (2)-(4);  we shall construct $x_{k+1}$ and $y_{k+1}$.
	
By (4) $f_k \in (x,y)^2 \subset (x,y)$, so there exist polynomials $A$ and $B$ such that $f_k = Ax + By$.
We can assume that there are no cancellations between $Ax$ and $By$, for instance by requiring that $B \in \CC[s_1, s_2, s_3, y]$.
Since by (3) $f_k \in R_{\geq k+1}$, we have  $A, B \in R_{\geq k+1}$.
Let $a$ (resp.\ $b$) the homogeneous part of $A$ (resp.\ $B$) of degree $k+1$. So $a, b \in R_{k+1}$ and $A-a, B-b \in R_{\geq k+2}$.
Moreover, since $f_k \in (x,y)^2$ we have $A, B \in (x,y)$, and consequently $a,b \in (x,y)$.

Set
\begin{equation*}
	x_{k+1} = x_k - b \quad \text{and} \quad y_{k+1} = y_k - a.
\end{equation*}
By the inductive hypothesis we have $x_k, y_k \in R_{\leq k} \cap (x,y)$ and by construction we have $a,b \in R_{k+1} \cap (x,y)$; therefore $x_{k+1}$ and $y_{k+1}$ lie in $R_{\leq k+1} \cap (x,y)$; this is (2) for $k+1$.
Clearly we also have (3) for $k+1$.

We have
\begin{align*}
	f_{k+1} &= x_{k+1} y_{k+1} - (xy + s_1 x^2 + s_2 y^2 + s_3 x^2 y^2) \\
	&= (x_k - b)(y_k - a) - (xy + s_1 x^2 + s_2 y^2 + s_3 x^2 y^2) \\
	&= f_k -b y_k - a x_k +ab \\
	&= (Ax + By) - a x_k - b y_k +ab.
\end{align*}
Since $A,B,a,b,x_k,y_k \in (x,y)$ we get $f_{k+1} \in (x,y)^2$.
Moreover, we can write
\[
f_{k+1} = a(x-x_k) + (A-a) x + b(y-y_k) + (B-b) y + ab.
\]
Since $a,b \in R_{k+1}$, $A-a, B-b \in R_{\geq k+2}$ and $x-x_k, y-y_k \in R_{\geq 1}$, we have $f_{k+1} \in R_{\geq k+2}$. We have obtained (4) for $k+1$.
\end{proof}

\begin{lemma} \label{lemma:two_deformations_are_isomorphic}
Consider the affine $3$-fold $U = \Spec \CC [ x,y,z,w]  / (xy-zw)$.
Then the two formal deformations of $U$ over $\Spf \CC \pow{s_1, s_2, s_3, s_4}$ given by
\[
xy-zw + s_1 x^2 + s_2 y^2 + s_3 x^2 y^2 + s_4 = 0
\]
and by
\[
xy-zw + s_4 =0,
\]
respectively,
are isomorphic.
\end{lemma}

\begin{proof}
Fix $k \in \NN$. Consider the finite $\CC$-algebra
\[
S_k = \CC[s_1, s_2, s_3, s_4]/(s_1, s_2, s_3, s_4)^{k+1}
\]
 and the following two $S_k$-algebras
\begin{align*}
		A_k &= S_k[x,y,z,w] / (xy-zw + s_4), \\
	B_k &= S_k[x,y,z,w] / (xy-zw + s_1 x^2 + s_2 y^2 + s_3 x^2 y^2 + s_4).
\end{align*}
Let $x_k$ and $y_k$ be the polynomials constructed in Lemma~\ref{lemma:sequence_of_polynomials}.
By (4) one can
consider the $S_k$-algebra homomorphism $\phi_k \colon A_k \to B_k$ given by $x \mapsto x_k$, $y \mapsto y_k$, $z \mapsto z$, $w \mapsto w$.
By (1) and (2) we have $x - x_k \in (s_1, s_2, s_3, s_4)$, so $\phi_k$ is just a translation, hence it is invertible.

By (3) we have
\[
x_k \equiv x_{k-1} \mod (s_1, s_2, s_3, s_4)^k \quad \text{and} \quad y_k \equiv y_{k-1} \mod (s_1, s_2, s_3, s_4)^k.
\]
So $\phi_k$ and $\phi_{k-1}$ are equal modulo $(s_1, s_2, s_3, s_4)^k$. More precisely, by using the natural projection $S_k \onto S_{k-1}$ we get that $\phi_k \otimes_{S_k} \mathrm{id}_{S_{k-1}}$ and $\phi_{k-1}$ are equal as isomorphisms from $A_{k-1} = A_k \otimes_{S_k} S_{k-1}$ to $B_{k-1} = B_k \otimes_{S_k} S_{k-1}$.
Therefore the system of isomorphisms $\phi_k$'s gives the required isomorphism of formal deformations of $U$.
\end{proof}

\begin{lemma}\label{lemma:discriminant_locus_def_3ODP}
Consider the affine $3$-fold $U = \Spec \CC [ x,y,z,w]  / (xy-zw)$ and consider the flat deformation $\cU \to \AA^4 = \Spec \CC [s_1, s_2, s_3, s_4]$ of $U$ given by the equation
\begin{equation} \label{eq:deformed}
xy-zw + s_1 x^2 + s_2 y^2 + s_3 x^2 y^2 + s_4 = 0.
\end{equation}

Then the discriminant locus of $\cU \to \AA^4$, i.e.\ the locus in $\AA^4$ over which the fibres of $\cU \to \AA^4$ are singular, is the vanishing locus of
\begin{equation} \label{eq:discriminant}
s_4 (16 s_1^2 s_2^2 - 32 s_1 s_2 s_3 s_4 - 8 s_1 s_2 + 16 s_3^2 s_4^2 - 8 s_3 s_4 + 1).
\end{equation}
In particular, in a neighbourhood of $0$ in $\AA^4$ the discriminant locus coincides with the hyperplane $\{s_4 = 0\}$.
\end{lemma}

\begin{proof}
	Let $f \in \CC[s_1, s_2, s_3, s_4, x, y, z, w]$ be the polynomial appearing in \eqref{eq:deformed}.
	Consider the ideal $J \subseteq \CC[s_1, s_2, s_3, s_4, x, y, z, w]$ generated by $f$ and by the $4$ partial derivatives of $f$ with respect to $x$, $y$, $z$  and $w$.
	By the Jacobian criterion and \cite[\S3.2]{cox_little_oshea}, the discriminant locus in $\AA^4 = \Spec \CC [s_1,s_2,s_3,s_4]$ is defined by the ideal $J \cap \CC[s_1, s_2, s_3, s_4]$. By \cite[Theorem~3.1.2]{cox_little_oshea} this last ideal can be easily computed with any computer algebra software by taking a Gr\"obner basis of $J$ with respect to the lexicographic monomial order $>$ such that $x > y>z>w>s_1>s_2>s_3>s_4$.
	The ideal $J \cap \CC[s_1, s_2, s_3, s_4]$ is generated by the polynomial in \eqref{eq:discriminant}.
	
	The last assertion in Lemma~\ref{lemma:discriminant_locus_def_3ODP} is obvious because the second factor of the polynomial in \eqref{eq:discriminant} does not vanish at the origin.
\end{proof}

\begin{proof}[Proof of Proposition~\ref{prop:deforming_X_inside_F}]

We have already observed that Proposition~\ref{prop:X_inside_F} says that the fibre of $\cX \to \AA^4$ over the origin is $X$.
Moreover, one can easily see that all monomials in the Cox coordinates of $F$ appearing in \eqref{eq:equation_of_X_inside_F} have the same degree with respect to the grading given by $\Pic(F)$, i.e.\ they are all sections of the line bundle $\cO_F(0,0,2)$ on $F$. In other words, the fibres of $\cX \to \AA^4$ are elements of the linear system $\vert \cO_F(X) \vert$.

Consider the formal deformation of $X$ over $\Spf \CC \pow{c_1, c_2, c_3, c_4}$ given by taking the formal completion of $\cX \to \AA^4$ at the central fibre.

We know that the infinitesimal deformations of $X$ are exactly the infinitesimal deformations of its singularities, i.e.\ of the $4$ affine open subset $U_\al$, $U_\beta$, $U_\gamma$, $U_\delta$ which contain the $4$ singular points of $X$. Indeed, the map in \eqref{eq:restriction_map_def_functors} is smooth and induces an isomorphism on tangent spaces.
Therefore, in order to check that our formal deformation of $X$ over $\Spf \CC \pow{c_1, c_2, c_3, c_4}$ is the miniversal deformation of $X$, we need to check that it induces the miniversal deformation of $U_\al$, $U_\beta$, $U_\gamma$, and $U_\delta$ when restricted to $U_\al$, $U_\beta$, $U_\gamma$, and $U_\delta$, respectively.

More precisely, this deformation of $X$ over $\Spf \CC \pow{c_1, c_2, c_3, c_4}$ must be induced by the miniversal deformation of $X$, which is over $\Spf \CC \pow{t_\al, t_\beta, t_\gamma, t_\delta}$, via a local $\CC$-algebra homomorphism
\begin{equation} \label{eq:ring_homomorphism_versal}
\psi \colon \CC \pow{t_\al, t_\beta, t_\gamma, t_\delta} \longrightarrow  \CC \pow{c_1, c_2, c_3, c_4}.
\end{equation}
We will need to prove that this ring homomorphism is an isomorphism.

Before doing that, recall all notation from Proposition~\ref{prop:definition_P_X} and Proposition~\ref{prop:X_inside_F}. Denote by $C_{ijkl}$ the cone in $\Sigma_F$ generated by $r_i, r_j, r_k$ and $r_l$.

\medskip

Let us start from $U_\al$. It is easy to check that the linear map $A \colon N_X \to N_F$ maps the cone $\sigma_\alpha$ to $C_{1257}$. Therefore, the toric morphism $X \to F$ associated to $A$ maps $U_\al$ to the affine chart $U_{1257}^F \subset F$ associated to the cone $C_{1257}$.
We know that the monoid $\sigma_\al^\vee \cap M_X$ is generated by 
\begin{equation*}
						x_\al = (-1,1,1) \quad y_\al = (1,0,0) \quad z_\al = (0,0,1) \quad w_\al = (0,1,0),
\end{equation*}
whereas one can check that the monoid $C_{1257}^\vee \cap M_F$ is generated by
\begin{equation*}
u_1 =	(1,0,0,0) \quad u_2 = (0,1,0,0) \quad u_5 = (0,0,1,0) \quad u_7 = (0,0,0,1).
\end{equation*}
Here we have used the names of the Cox coordinates of $F$ because these $4$ elements correspond to the dehomogenisations of the Cox coordinates of $F$ with respect to the cone $C_{1257}$. Indeed $U_{1257}^F$ is the locus in $F$ where $u_3 \neq 0$, $u_4 \neq 0$ and $u_6 \neq 0$.
It is easy to see that the transpose $A^T \colon M_F \to M_X$ acts as
\begin{equation} \label{eq:assignements_chart_U_alpha}
	u_1 \mapsto w_\al \quad u_2 \mapsto z_\al \quad u_5 \mapsto x_\al \quad u_7 \mapsto y_\al.
\end{equation}
These are exactly the assignments of the surjective $\CC$-algebra homomorphism associated to the closed embedding
\[
U_\al = \Spec \frac{\CC [ x_\al, y_\al, z_\al, w_\al] }{(x_\al y_\al - w_\al z_\al)} \longrightarrow U_{1257}^F = \Spec \CC[u_1, u_2, u_5, u_7]
\]
which is the restriction of the toric closed embedding $X \into F$.
If we dehomogenise \eqref{eq:equation_of_X_inside_F} with respect to $C_{1257}$ we get the equation
\begin{equation*} 
	u_5 u_7-u_1u_2 + c_1 u_1^2u_2^2+c_2 u_1^2 + c_3 u_2^2 + c_4 =0.
\end{equation*}
If now we apply \eqref{eq:assignements_chart_U_alpha}, we deduce that the flat family $\cX \to \AA^4$ induces the deformation of $U_\al$ over $\Spf \CC \pow{c_1, c_2, c_3, c_4}$ given by the equation
\begin{equation} \label{eq:deformed_equation_U_alpha}
		x_\al y_\al - w_\al z_\al + c_1 w_\al^2 z_\al^2+c_2 w_\al^2 + c_3 z_\al^2 + c_4 =0.
\end{equation}
By Lemma~\ref{lemma:two_deformations_are_isomorphic} this deformation of $U_\al$ over $\Spf \CC \pow{c_1, c_2, c_3, c_4}$ is isomorphic to the one given by the equation 
\begin{equation*}
	x_\al y_\al - w_\al z_\al + c_4 =0.
\end{equation*}
By comparing with \eqref{eq:meaning_of_t_alpha} we have that the homomorphism $\psi$ in \eqref{eq:ring_homomorphism_versal} satisfies $\psi(t_\al) = c_4$.

\medskip

Let us now look at $U_\beta$. The analysis is very similar. The morphism associated to $A$ maps $U_\beta$ to $U^F_{1457}=\{u_2\neq 0,\, u_3\neq 0,\, u_6\neq 0\}$ and the induced map on the respective $\CC$-algebras is 
\begin{equation} \label{eq:assignements_chart_U_beta}
	u_1 \mapsto y_\beta \quad u_4 \mapsto x_\beta \quad u_5 \mapsto z_\beta \quad u_7 \mapsto w_\beta.
\end{equation}
In the chart $U^F_{1457}$, equation \eqref{eq:equation_of_X_inside_F} becomes
\begin{equation*} 
	u_5 u_7-u_1u_4 + c_1 u_1^2+c_2 u_1^2u_4^2 + c_3 + c_4 u_4^2 =0,
\end{equation*}
and the flat family $\cX \to \AA^4$ induces the deformation of $U_\beta$ over $\Spf \CC \pow{c_1, c_2, c_3, c_4}$ given by
\begin{equation} \label{eq:deformed_equation_U_beta}
		z_\beta w_\beta - x_\beta y_\beta + c_1y_\beta^2+c_2 x_\beta^2y_\beta^2 + c_3+ c_4x_\beta^2 =0.
\end{equation}
By Lemma~\ref{lemma:two_deformations_are_isomorphic} this deformation is isomorphic to 
\begin{equation*}
	z_\beta w_\beta - x_\beta y_\beta + c_3 =0, 
\end{equation*}
therefore we have that $\psi(t_\beta)=c_3$.

Similarly, $U_\gamma$ maps to $U^F_{3457}$ such that   
\begin{equation} \label{eq:assignements_chart_U_gamma}
	u_3 \mapsto x_\gamma \quad u_4 \mapsto y_\gamma  \quad u_5 \mapsto z_\gamma  \quad u_7 \mapsto w_\gamma
\end{equation}
Equation \eqref{eq:equation_of_X_inside_F} restricted to $U^F_{3457}$ induces the deformation of $U_\gamma$ given by
\begin{equation} \label{eq:deformed_equation_U_gamma}
		z_\gamma w_\gamma - x_\gamma y_\gamma + c_1+c_2 y_\gamma^2 + c_3 x_\gamma^2+ c_4x_\gamma^2y_\gamma^2 =0.
\end{equation}
We use Lemma~\ref{lemma:two_deformations_are_isomorphic} to deduce that $\psi(t_\gamma)=c_1$. 

Lastly, the map $U_\delta\to U^F_{2357}$ given by
\begin{equation} \label{eq:assignements_chart_U_delta}
	u_2 \mapsto x_\delta \quad u_3 \mapsto y_\delta  \quad u_5 \mapsto z_\delta  \quad u_7 \mapsto w_\delta 
\end{equation} induces the deformation of $U_\delta$ over $\Spf \CC \pow{c_1, c_2, c_3, c_4}$  
\begin{equation} \label{eq:deformed_equation_U_delta}
		z_\delta w_\delta - x_\delta y_\delta + c_1x_\delta^2+c_2 + c_3 x_\delta^2y_\delta^2+ c_4y_\delta^2 =0.
\end{equation} 
Again the change of variables in Lemma~\ref{lemma:two_deformations_are_isomorphic} allows us to conclude that $\psi(t_\delta)=c_2$.

%
%
%
%
\medskip

We have proved that the ring homomorphism $\psi$ in \eqref{eq:ring_homomorphism_versal} satisfies $\psi(t_\al) = c_4$, $\psi(t_\beta) = c_3$, $\psi(t_\gamma) = c_1$ and $\psi(t_\delta) = c_2$. Therefore $\psi$ is an isomorphism. This proves that the considered deformation of $X$ is the miniversal one.
This concludes the proof of the first part of the assertions in Proposition~\ref{prop:deforming_X_inside_F}.

\medskip

We need to analyse the discriminant locus of the family $\cX \to \AA^4$.
In a neighbourhood of the origin in $\AA^4 = \Spec \CC[c_1, c_2, c_3, c_4]$, the discriminant locus coincides with the union of the discriminant locus of the deformation of the $4$ affine charts of $X$ which contain the singularities of $X$.

We have computed the equation of the deformation of $U_\al$ in \eqref{eq:deformed_equation_U_alpha}. By Lemma~\ref{lemma:discriminant_locus_def_3ODP} the discriminant locus of this deformation of $U_\al$ coincides with the hyperplane $\{c_4 = 0\}$ in a neighbourhood of the origin in $\AA^4$.

By repeating the argument for $U_\beta$, $U_\gamma$ and $U_\delta$ and by taking the union of these hyperplanes, we conclude.
\end{proof}

\begin{remark}
By analysing certain automorphisms of $F$ which leave $X$ invariant, it might be possible that the discussion on the deformation of $U_\alpha$ be sufficient in order to deduce the results about the deformations of $U_\beta$, $U_\gamma$, $U_\delta$  in the proof above. This approach might make the proof less repetitive, but not shorter.
\end{remark}

\section{Conclusion}

\begin{proof}[Proof of Theorem~\ref{thm:main}]
Consider the flat deformation
	\[
	\begin{tikzcd}
		X \arrow{d} \arrow[hook]{r} &\cX \arrow{d}\\
		0 \arrow[hook]{r}  & \AA^4
	\end{tikzcd}
	\]
of $X$ considered in Proposition~\ref{prop:deforming_X_inside_F}.
By Proposition~\ref{prop:definition_P_X}(1) $X$ is K-polystable, in particular K-semistable.
Since being K-semistable is an open property \cite{BLX,xu_minimizing}, there exists an open neighbourhood $W$ of the origin in $\AA^4$ such that each fibre of $\cX_W := \cX \times_{\AA^4} W \to W$ is K-semistable.
Let $\Kss{3}{28}$ be the algebraic stack which parametrises $\QQ$-Gorenstein families of $3$-dimensional K-semistable Fano varieties with anticanonical volume $28$ \cite{xu_survey}.
The family $\cX_W \to W$ is induced by a morphism $W \to \Kss{3}{28}$.

The K-moduli space $ M := \Kps{3}{28}$ is the good moduli space of $\Kss{3}{28}$ in the sense of Alper, in particular there is a natural morphism $\Kss{3}{28} \to M$.
By composition we get a morphism $W \to M$ between schemes of finite type over $\CC$.
Since $\cX_W \to W$ is essentially the miniversal deformation of $X$, we have that in an analytic neighbourhood of $[X]$ the morphism $W \to M$ behaves as the map $\Delta^4 \to \Delta^2$ between polydiscs described in the proof of Proposition~\ref{prop:definition_P_X}(9).
In particular this shows that $W$ dominates the irreducible component of $M$ containing $[X]$.
Since $W$ is rational, the irreducible component of $M$ containing $[X]$ is unirational.

However, we already know by Corollary~\ref{cor:almost_everything_without_rationality} that there exists an open neighbourhood $U$ of $[X]$ in $M$ which is a smooth surface. Hence $U$ must be unirational, and hence rational because $\dim U = 2$.

By Corollary~\ref{cor:almost_everything_without_rationality} we know that the discriminant locus in $U$ (i.e.\ the locus in $U$ parametrising singular Fano $3$-folds) is a smooth curve.
This locus is dominated by the discriminant locus in $W$, which by Proposition~\ref{prop:deforming_X_inside_F} coincides with the reducible divisor $\{ c_1 c_2 c_3 c_4 = 0 \}$ in a neighbourhood of $0$ in $\AA^4$.
Since this divisor has rational components, we conclude that the discriminant locus in $U$ is a smooth rational curve.
\end{proof}


\section{Appendix: Laurent inversion and global embeddings}
\label{LIsec}

The toric Fano $3$-fold considered in Proposition~\ref{prop:definition_P_X} can be embedded into a toric $5$-fold as a complete intersection. This is done via the Laurent inversion method, developed by Coates--Kasprzyk--Prince \cite{laurent_inversion, from_cracked_to_fano, cracked_fano_toric_ci}.
Since one of the obtained equations is a linear cone, we can then reduce to the hypersurface embedding that appears in Proposition~\ref{prop:X_inside_F}.

\begin{prop} \label{prop:X_inside_Y}
	Let $X$ be the toric Fano $3$-fold considered in Proposition~\ref{prop:definition_P_X}.
	Let $Y$ be the toric $5$-fold given by the GIT quotient $\AA^8 \git (\CC^*)^3$ with stability condition $\omega=(1,1,1)$ and the following weight matrix 
	\[\begin{matrix}
y_1& y_2& y_3&y_4&y_5&y_6&y_7&y_8\\
\hline 
1 & 0 & 0 & 1 & 0& 0& -1& 0\\
0 &1 &0 &0 &1 &0 &-1& 0\\
0 &0 &1 &0 &0 &1 &1 &1\\
\end{matrix}\]
 Then $X$ can be embedded into $Y$ as the zero-section of a section of $L \oplus L^{\otimes 2}$, where $L$ is the $(\CC^*)^3$-linearised line bundles on $Y$ of weights $(0,0,1)$.
 More specifically, the equations of $X$ in the Cox coordinates of $Y$ are
\[y_4y_5y_7=y_3  \quad \text{and} \quad y_6y_8=y_7y_1y_2y_3.\]
\end{prop}

\begin{proof}
This is an immediate application of \cite[Algorithm 5.1]{laurent_inversion}. Consider the following $3$ polytopes:
\begin{align*}
	S_1 &= \conv \{ \rho_2, \rho_4 \}, \\
	S_2 &= \conv \{\rho_3, \rho_5 \}, \\
	S_3 &= \conv \{\rho_1, \rho_6, \rho_7, \rho_8, \rho_9, \rho_{10} \},
\end{align*}
which are depicted in Figure~\ref{ScaffoldingP735}.

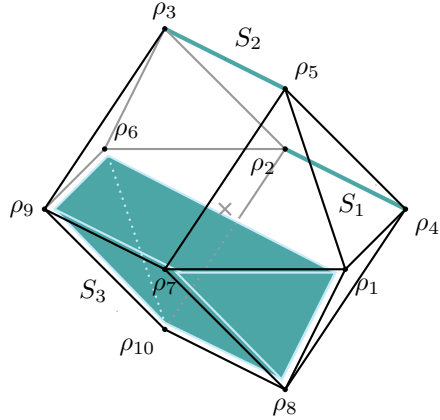
\begin{figure}[ht]
\centering
\begin{tikzpicture}[scale=0.8,every node/.style={scale=1}]

\filldraw[color=cyan!20, fill=blue!50!green!70, thick] (-2+1/18,1-2/18) -- (2-3/18,-1-1/18) -- (1-1/18,-3+3/18) -- (-1,-2+2/18) -- (-3+3/18,0) -- cycle; 
\draw[ultra thick, blue!50!green!70] (1,1) -- (3,0);
\draw[ultra thick, blue!50!green!70] (-1,3) -- (1,2);
\draw[thick, dotted, cyan!20] (-2+1/18,1-2/18) --  (-1,-2+2/18);

\draw[gray!80, thick] (1,1) -- (-2,1);
\draw[gray!80, thick] (1,1) -- (-1,3);
\draw[gray!80, thick] (-1,3) -- (-2,1);
\draw[gray!80, thick] (-3,0) -- (-2,1);
\draw[gray!80, thick] (1,1) -- (1/4,-1/8);
\draw[gray!80, dotted, thick] (1/4,-1/8) -- (-1,-2);

\draw[thick, cyan!20] (2-3/18,-1-1/18) -- (-1+3.5/18,-1-1/18) --   (1-1/18,-3+3/18)  -- cycle;
\draw[thick, cyan!20]  (-3+3/18,0)  -- (-1+5/18,-1-1/18);

\draw[black, thick] (-3,0) -- (-1,3);
\draw[black, thick] (-3,0) -- (-1,-1);
\draw[black, thick] (-1,-1) -- (1,2);
\draw[black, thick] (1,2) -- (3,0);
\draw[black, thick] (1,2) -- (2,-1);
\draw[black, thick] (-1,-1) -- (2,-1);
\draw[black, thick] (2,-1) -- (3,0);
\draw[black, thick] (3,0) -- (1,-3);
\draw[black, thick] (2,-1) -- (1,-3);
\draw[black, thick] (-1,-2) -- (1,-3);
\draw[black, thick] (-1,-2) -- (-3,0);
\draw[black, thick] (-1,-1) -- (1,-3);

\draw[gray!80, thick] (0.1,0.1) -- (-0.1,-0.1);
\draw[gray!80, thick] (-0.1,0.1) -- (0.1,-0.1);

\filldraw[black] (2,-1) circle (1pt) node[anchor=north west] {$\rho_1$};
\filldraw[black] (3,0) circle (1pt) node[anchor=north west] {$\rho_4$};
\filldraw[black] (1,1) circle (1pt) node[anchor=north east] {$\rho_2$};
\filldraw[black] (-2,1) circle (1pt) node[anchor=south west] {$\rho_6$};
\filldraw[black] (-3,0) circle (1pt) node[anchor=east] {$\rho_9$};
\filldraw[black] (1,-3) circle (1pt) node[anchor=north] {$\rho_8$};
\filldraw[black] (-1,-2) circle (1pt) node[anchor=north east] {$\rho_{10}$};
\filldraw[black] (1,2) circle (1pt) node[anchor=south west] {$\rho_5$};
\filldraw[black] (-1,3) circle (1pt) node[anchor=south] {$\rho_3$};
\filldraw[black] (-1,-1) circle (1pt) node[anchor=north ] {$\rho_7$};

\filldraw[black] (2.5,0.4) circle (0pt) node[anchor=north east] {$S_1$};
\filldraw[black] (0,2.5) circle (0pt) node[anchor=south west] {$S_2$};
\filldraw[black] (-1.8,-1.7) circle (0pt) node[anchor=south east] {$S_3$};

%
\end{tikzpicture}

\caption{Scaffolding of P735}
\label{ScaffoldingP735}

\end{figure}

Let $\Sigma_Z$ be the normal fan of $S_3$. The rays of $\Sigma_Z$ are the following vectors in $M = \Hom_\ZZ(N, \ZZ)$:
\begin{equation*}
	\begin{aligned}
		v_1 & = (0,-1,0), \\
		v_2 &= (0,0,-1), \\
		v_3 &= (-1,1,1), \\
	v_4 &=	(0,1,1), \\
	v_5 &=	(1,0,0).
	\end{aligned}
\end{equation*}
Let $Z$ be the $T_M$-toric variety associated to the $\Sigma_Z$. It is smooth, projective of dimension $3$. Since $\conv(S_1 \cup S_2 \cup S_3)=P$, the collection $S=\{S_1, S_2, S_3\}$ represents a \textit{scaffolding} of $P$ with shape $Z$, as introduced in \cite{laurent_inversion}. 

%
We also see that $Z$ is a $\PP^2$-bundle over $\PP^1$, namely $Z = \PP_{\PP^1} \left( \cO \oplus \cO(1)^{\oplus 2}  \right)$. This ensures that the Laurent inversion construction embeds $X$ as a complete intersection of codimension $\rho(Z)=2$ (see the proof of \cite[Proposition~12.2]{laurent_inversion}). 
In particular, the relations that give the $\PP^2$-bundle structure on $\PP^1$ are 
\begin{equation}
\label{eq:relations}
	\begin{aligned}
		v_1+v_2+v_4 &=0\\
		v_3+v_5&=v_4.
	\end{aligned}
\end{equation}
Here the first relation simply describes the $\PP^2$ fibre, while the second is a primitive relation (in the sense of Batyrev \cite{Batyrev}) obtained by pulling back of the relation $v'_3+v'_5=0$ in the quotient lattice induced by the projection to $\PP^1$. 

Consider the following $3$ torus invariant nef divisors on $Z$:
\begin{equation}
\label{eq:struts}
	\begin{aligned}
		D_1 &= E_1 - E_4, \\
		D_2 &= E_2 - E_4, \\
		D_3 &= E_3 + E_4 + E_5,
	\end{aligned}
\end{equation} 
where $E_i$ are the torus-invariant divisors associated to the $v_i$. Their section polytopes are exactly $S_1$, $S_2$, $S_3$, respectively.

We consider the lattice $\widetilde{N} = \Div_{T_M} Z= \ZZ^5$ with basis $E_1, \dots, E_5$. This will be the ambient lattice for the fan of $Y$ which we describe below as a GIT quotient using the expressions in \eqref{eq:struts}. We then use the relations \eqref{eq:relations} to deduce the equations of $X$ inside $Y$, keeping in mind the correspondence $E_i \leftrightarrow v_i$.

Following the construction in \cite[Algorithm~5.1]{laurent_inversion}, the weight matrix of $Y$ is of size $(r+z)\times r=8\times 3$, where $r$ is the number of $S_i$ and $z$ is the number of rays in $\Sigma_Z$, given by:
\[\begin{matrix}
I_1 &I_2 &I_3& E_1 &E_2 &E_3 &E_4 & E_5\\
\hline 
1 & 0 & 0 & 1 & 0& 0& -1& 0\\
0 &1 &0 &0 &1 &0 &-1& 0\\
0 &0 &1 &0 &0 &1 &1 &1\\
\hline
y_1& y_2& y_3&y_4&y_5&y_6&y_7&y_8
\end{matrix}\]
Here a row $i$ corresponds to a strut $S_i$ (or, alternatively, to a divisor $D_i$ whose polytope is $S_i$) and each $(i,j)$ entry outside of the $r\times r$-identity block is the coefficient of $E_{j}$ in the expression of $D_i$ in \eqref{eq:struts}. 

The equations of $X$ in the Cox coordinates on $Y$ are \[y_4y_5y_7=y_3 \quad \text{and} \quad y_6y_8=y_7y_1y_2y_3,\] which are deduced from \eqref{eq:relations} as explained in \cite[Proposition~12.2]{laurent_inversion}.



In particular, this shows that $X$ is a complete intersection $L_1\oplus L_2$, where $L_i$ are $(\CC^*)^3$-linearised line bundles with weights
\[
L_1=\cO_Y(y_4)\otimes \cO_Y(y_5)\otimes\cO_Y(y_7)=\cO_Y\begin{pmatrix}0\\0\\1\end{pmatrix}
\]
and
\[
L_2=\cO_Y(y_6)\otimes \cO_Y(y_8)=\cO_Y\begin{pmatrix}0\\0\\2\end{pmatrix}.
\]
According to \cite[Theorem~5.5]{laurent_inversion}, this is an embedding of a Fano variety; a posteriori the stability condition needed to define $Y$ should be chosen such that \[-K_Y-L_1-L_2=\cO_Y\left(\begin{pmatrix}1\\1\\4\end{pmatrix}-\begin{pmatrix}0\\0\\1\end{pmatrix}-\begin{pmatrix}0\\0\\2\end{pmatrix}\right)=\cO_Y\begin{pmatrix}1\\1\\1\end{pmatrix}\] is ample, i.e. $\omega=(1,1,1)$. \end{proof}

\begin{remark}
The ambient space $Y$ thus obtained is in fact a smooth toric $5$-fold. Since the first equation of the embedding in Proposition~\ref{prop:X_inside_Y} is linear in $y_3$, we can solve for this variable and transform the second equation into $y_6y_8=y_7^2y_1y_2y_4y_5$. Up to renaming the ambient Cox variables, this is exactly the embedding of $X$ inside the $4$-fold $F$ from Proposition~\ref{prop:X_inside_F}.
\end{remark}

\bibliography{Biblio_Liana_Andrea}

\begin{bibdiv}
\begin{biblist}

\bib{quarticthreefolds}{unpublished}{
      author={Abban, Hamid},
      author={Cheltsov, Ivan},
      author={Kasprzyk, Alexander},
      author={Liu, Yuchen},
      author={Petracci, Andrea},
       title={On {K}-moduli of quartic threefolds},
        date={2023},
        note={\arxiv{2210.14781} (to appear in {A}lgebr. {G}eom.)},
}

\bib{abban_zhuang}{article}{
      author={Abban, Hamid},
      author={Zhuang, Ziquan},
       title={K-stability of {F}ano varieties via admissible flags},
        date={2022},
     journal={Forum Math. Pi},
      volume={10},
       pages={Paper No. e15, 43},
         url={https://doi.org/10.1017/fmp.2022.11},
}

\bib{ABHLX}{article}{
      author={Alper, Jarod},
      author={Blum, Harold},
      author={Halpern-Leistner, Daniel},
      author={Xu, Chenyang},
       title={Reductivity of the automorphism group of {K}-polystable {F}ano
  varieties},
        date={2020},
     journal={Invent. Math.},
      volume={222},
      number={3},
       pages={995\ndash 1032},
}

\bib{luna_etale_slice_stacks}{article}{
      author={Alper, Jarod},
      author={Hall, Jack},
      author={Rydh, David},
       title={A {L}una \'{e}tale slice theorem for algebraic stacks},
        date={2020},
     journal={Ann. of Math. (2)},
      volume={191},
      number={3},
       pages={675\ndash 738},
}

\bib{altmann_t1}{article}{
      author={Altmann, Klaus},
       title={Computation of the vector space {$T^1$} for affine toric
  varieties},
        date={1994},
        ISSN={0022-4049},
     journal={J. Pure Appl. Algebra},
      volume={95},
      number={3},
       pages={239\ndash 259},
}

\bib{calabi_problem_book}{book}{
      author={Araujo, Carolina},
      author={Castravet, Ana-Maria},
      author={Cheltsov, Ivan},
      author={Fujita, Kento},
      author={Kaloghiros, Anne-Sophie},
      author={Martinez-Garcia, Jesus},
      author={Shramov, Constantin},
      author={S{\"u}ss, Hendrik},
      author={Viswanathan, Nivedita},
       title={The {C}alabi problem for {F}ano threefolds},
      series={London Mathematical Society Lecture Note Series},
   publisher={Cambridge University Press},
        date={2023},
}

\bib{Batyrev}{article}{
      author={Batyrev, Victor~V.},
       title={{On the classification of smooth projective toric varieties}},
        date={1991},
     journal={Tohoku Mathematical Journal},
      volume={43},
      number={4},
       pages={569 \ndash  585},
         url={https://doi.org/10.2748/tmj/1178227429},
}

\bib{BatyrevSelivanova}{article}{
      author={Batyrev, Victor~V.},
      author={Selivanova, Elena~N.},
       title={Einstein-{K}\"ahler metrics on symmetric toric {F}ano manifolds},
        date={1999},
        ISSN={0075-4102,1435-5345},
     journal={J. Reine Angew. Math.},
      volume={512},
       pages={225\ndash 236},
}

\bib{belousov_loginov}{article}{
      author={Belousov, Grigory},
      author={Loginov, Konstantin},
       title={K-stability of {F}ano threefolds of rank 4 and degree 24},
        date={2023},
     journal={Eur. J. Math.},
      volume={9},
      number={3},
       pages={Paper No. 80, 20},
}

\bib{berman_polystability}{article}{
      author={Berman, Robert~J.},
       title={K-polystability of {${\Bbb Q}$}-{F}ano varieties admitting
  {K}\"{a}hler-{E}instein metrics},
        date={2016},
     journal={Invent. Math.},
      volume={203},
      number={3},
       pages={973\ndash 1025},
}

\bib{properness_K_moduli}{article}{
      author={Blum, Harold},
      author={Halpern-Leistner, Daniel},
      author={Liu, Yuchen},
      author={Xu, Chenyang},
       title={On properness of {K}-moduli spaces and optimal degenerations of
  {F}ano varieties},
        date={2021},
     journal={Selecta Math. (N.S.)},
      volume={27},
      number={4},
       pages={Paper No. 73, 39},
}

\bib{BlumJonsson}{article}{
      author={Blum, Harold},
      author={Jonsson, Mattias},
       title={Thresholds, valuations, and {K}-stability},
        date={2020},
        ISSN={0001-8708},
     journal={Advances in Mathematics},
      volume={365},
       pages={107062},
  url={https://www.sciencedirect.com/science/article/pii/S0001870820300888},
}

\bib{BLX}{article}{
      author={Blum, Harold},
      author={Liu, Yuchen},
      author={Xu, Chenyang},
       title={Openness of {K}-semistability for {F}ano varieties},
        date={2022},
     journal={Duke Math. J.},
      volume={171},
      number={13},
       pages={2753\ndash 2797},
}

\bib{blum_xu_uniqueness}{article}{
      author={Blum, Harold},
      author={Xu, Chenyang},
       title={Uniqueness of {K}-polystable degenerations of {F}ano varieties},
        date={2019},
     journal={Ann. of Math. (2)},
      volume={190},
      number={2},
       pages={609\ndash 656},
}

\bib{Magma}{article}{
      author={Bosma, Wieb},
      author={Cannon, John},
      author={Playoust, Catherine},
       title={The {M}agma algebra system. {I}. {T}he user language},
        date={1997},
        ISSN={0747-7171},
     journal={J. Symbolic Comput.},
      volume={24},
      number={3-4},
       pages={235\ndash 265},
        note={Computational algebra and number theory (London, 1993)},
}

\bib{GRDB-toric3}{misc}{
      author={Brown, Gavin},
      author={Kasprzyk, Alexander~M.},
       title={The {G}raded {R}ing {D}atabase},
         how={Database of toric canonical $3$-folds, online},
        note={\url{http://grdb.co.uk/forms/toricf3c}},
}

\bib{cheltsov_denisova_fujita}{article}{
      author={Cheltsov, Ivan},
      author={Denisova, Elena},
      author={Fujita, Kento},
       title={K-stable smooth {F}ano threefolds of {P}icard rank two},
        date={2024},
     journal={Forum Math. Sigma},
      volume={12},
       pages={Paper No. e41, 64},
}

\bib{cheltsov_fujita_kishimoto_okada}{article}{
      author={Cheltsov, Ivan},
      author={Fujita, Kento},
      author={Kishimoto, Takashi},
      author={Okada, Takuzo},
       title={K-stable divisors in {$\Bbb P^1\times\Bbb P^1\times\Bbb P^2$} of
  degree {$(1,1,2)$}},
        date={2023},
     journal={Nagoya Math. J.},
      volume={251},
       pages={686\ndash 714},
}

\bib{cheltsov2023kstability}{unpublished}{
      author={Cheltsov, Ivan},
      author={Guerreiro, Tiago~Duarte},
      author={Fujita, Kento},
      author={Krylov, Igor},
      author={Martinez-Garcia, Jesus},
       title={K-stability of {C}asagrande-{D}ruel varieties},
        date={2023},
        note={\arxiv{2309.12522} (to appear in {J.} {R}eine {A}ngew. {M}ath.)},
}

\bib{cheltsov_park}{article}{
      author={Cheltsov, Ivan},
      author={Park, Jihun},
       title={K-stable {F}ano threefolds of rank 2 and degree 30},
        date={2022},
     journal={Eur. J. Math.},
      volume={8},
      number={3},
       pages={834\ndash 852},
}

\bib{chen_donaldson_sun}{article}{
      author={Chen, Xiuxiong},
      author={Donaldson, Simon},
      author={Sun, Song},
       title={K\"{a}hler-{E}instein metrics on {F}ano manifolds. {I}, {II} and
  {III}},
        date={2015},
        ISSN={0894-0347},
     journal={J. Amer. Math. Soc.},
      volume={28},
      number={1},
       pages={183\ndash 278},
         url={https://doi.org/10.1090/S0894-0347-2014-00799-2},
}

\bib{ccggsk}{incollection}{
      author={Coates, Tom},
      author={Corti, Alessio},
      author={Galkin, Sergey},
      author={Golyshev, Vasily},
      author={Kasprzyk, Alexander},
       title={Mirror symmetry and {F}ano manifolds},
        date={2013},
   booktitle={European {C}ongress of {M}athematics},
   publisher={Eur. Math. Soc., Z\"{u}rich},
       pages={285\ndash 300},
}

\bib{quantum_periods}{article}{
      author={Coates, Tom},
      author={Corti, Alessio},
      author={Galkin, Sergey},
      author={Kasprzyk, Alexander},
       title={Quantum periods for 3-dimensional {F}ano manifolds},
        date={2016},
        ISSN={1465-3060},
     journal={Geom. Topol.},
      volume={20},
      number={1},
       pages={103\ndash 256},
         url={https://doi.org/10.2140/gt.2016.20.103},
}

\bib{laurent_inversion}{article}{
      author={Coates, Tom},
      author={Kasprzyk, Alexander},
      author={Prince, Thomas},
       title={Laurent inversion},
        date={2019},
     journal={Pure Appl. Math. Q.},
      volume={15},
      number={4},
       pages={1135\ndash 1179},
}

\bib{cox_little_oshea}{book}{
      author={Cox, David~A.},
      author={Little, John},
      author={O'Shea, Donal},
       title={Ideals, varieties, and algorithms},
     edition={Fourth},
      series={Undergraduate Texts in Mathematics},
   publisher={Springer, Cham},
        date={2015},
        note={An introduction to computational algebraic geometry and
  commutative algebra},
}

\bib{cox_little_schenck}{book}{
      author={Cox, David~A.},
      author={Little, John~B.},
      author={Schenck, Henry~K.},
       title={Toric varieties},
      series={Graduate Studies in Mathematics},
   publisher={American Mathematical Society, Providence, RI},
        date={2011},
      volume={124},
        ISBN={978-0-8218-4819-7},
         url={https://doi.org/10.1090/gsm/124},
}

\bib{Delcroix}{article}{
      author={Delcroix, Thibaut},
       title={K-stability of {F}ano spherical varieties},
        date={2020},
     journal={Annales {S}cientifiques de l'{{\'E}}cole {N}ormale
  {S}up{\'e}rieure},
      volume={53},
      number={3},
       pages={615\ndash 662},
}

\bib{denisova}{article}{
      author={Denisova, Elena},
       title={On {$K$}-stability of {$\Bbb P^3$} blown up along the disjoint
  union of a twisted cubic curve and a line},
        date={2024},
     journal={J. Lond. Math. Soc. (2)},
      volume={109},
      number={5},
       pages={Paper No. e12911, 26},
}

\bib{donaldson_stability}{article}{
      author={Donaldson, S.~K.},
       title={Scalar curvature and stability of toric varieties},
        date={2002},
     journal={J. Differential Geom.},
      volume={62},
      number={2},
       pages={289\ndash 349},
}

\bib{Galkin}{unpublished}{
      author={Galkin, Sergey},
       title={Small toric degenerations of {F}ano threefolds},
        date={2018},
         url={https://arxiv.org/abs/1809.02705},
        note={\arxiv{1809.02705}},
}

\bib{Heuberger_QFLI}{unpublished}{
      author={Heuberger, Liana},
       title={$\mathbb{Q}$-{F}ano threefolds and {L}aurent inversion},
        note={\arxiv{2202.04184} (to appear in {N}agoya {M}ath. {J}.)},
}

\bib{IltenSuess}{article}{
      author={Ilten, Nathan},
      author={S{\"u}{\ss}, Hendrik},
       title={{K-stability for Fano manifolds with torus action of complexity
  $1$}},
        date={2017},
     journal={Duke Mathematical Journal},
      volume={166},
      number={1},
       pages={177 \ndash  204},
         url={https://doi.org/10.1215/00127094-3714864},
}

\bib{iskovskikh_prokhorov}{book}{
      author={Iskovskikh, V.~A.},
      author={Prokhorov, Yu.~G.},
      editor={Shafarevich, I.~R.},
       title={Algebraic geometry. {V}},
      series={Encyclopaedia of Mathematical Sciences},
   publisher={Springer-Verlag, Berlin},
        date={1999},
      volume={47},
        ISBN={3-540-61468-0},
        note={Fano varieties, A translation of {{\i}t Algebraic geometry. 5}
  (Russian), Ross. Akad. Nauk, Vseross. Inst. Nauchn. i Tekhn. Inform., Moscow,
  Translation edited by A. N. Parshin and I. R. Shafarevich},
}

\bib{jiang_boundedness}{article}{
      author={Jiang, Chen},
       title={Boundedness of {$\Bbb Q$}-{F}ano varieties with degrees and
  alpha-invariants bounded from below},
        date={2020},
     journal={Ann. Sci. \'{E}c. Norm. Sup\'{e}r. (4)},
      volume={53},
      number={5},
       pages={1235\ndash 1248},
}

\bib{ask_petracci}{article}{
      author={Kaloghiros, Anne-Sophie},
      author={Petracci, Andrea},
       title={On toric geometry and {K}-stability of {F}ano varieties},
        date={2021},
     journal={Trans. Amer. Math. Soc. Ser. B},
      volume={8},
       pages={548\ndash 577},
}

\bib{Kasprzyk}{article}{
      author={Kasprzyk, Alexander},
       title={Canonical toric {F}ano threefolds},
        date={2010},
     journal={Canadian Journal of Mathematics},
      volume={62},
      number={6},
       pages={1293\ndash 1309},
}

\bib{kreuzer_skarke}{article}{
      author={Kreuzer, Maximilian},
      author={Skarke, Harald},
       title={Classification of reflexive polyhedra in three dimensions},
        date={1998},
        ISSN={1095-0761},
     journal={Adv. Theor. Math. Phys.},
      volume={2},
      number={4},
       pages={853\ndash 871},
         url={https://doi.org/10.4310/ATMP.1998.v2.n4.a5},
}

\bib{lwx}{article}{
      author={Li, Chi},
      author={Wang, Xiaowei},
      author={Xu, Chenyang},
       title={Algebraicity of the metric tangent cones and equivariant
  {K}-stability},
        date={2021},
     journal={J. Amer. Math. Soc.},
      volume={34},
      number={4},
       pages={1175\ndash 1214},
}

\bib{liu_cubic_fourfolds}{article}{
      author={Liu, Yuchen},
       title={K-stability of cubic fourfolds},
        date={2022},
        ISSN={0075-4102},
     journal={J. Reine Angew. Math.},
      volume={786},
       pages={55\ndash 77},
         url={https://doi.org/10.1515/crelle-2022-0002},
}

\bib{liu_rank2}{article}{
      author={Liu, Yuchen},
       title={K-stability of {F}ano threefolds of rank 2 and degree 14 as
  double covers},
        date={2023},
        ISSN={0025-5874},
     journal={Math. Z.},
      volume={303},
      number={2},
       pages={Paper No. 38, 9},
         url={https://doi.org/10.1007/s00209-022-03192-4},
}

\bib{liu_xu}{article}{
      author={Liu, Yuchen},
      author={Xu, Chenyang},
       title={K-stability of cubic threefolds},
        date={2019},
        ISSN={0012-7094},
     journal={Duke Math. J.},
      volume={168},
      number={11},
       pages={2029\ndash 2073},
         url={https://doi.org/10.1215/00127094-2019-0006},
}

\bib{projectivity_K_moduli_final}{article}{
      author={Liu, Yuchen},
      author={Xu, Chenyang},
      author={Zhuang, Ziquan},
       title={Finite generation for valuations computing stability thresholds
  and applications to {K}-stability},
        date={2022},
     journal={Ann. of Math. (2)},
      volume={196},
      number={2},
       pages={507\ndash 566},
}

\bib{Mabuchi}{article}{
      author={Mabuchi, Toshiki},
       title={{Einstein-K{\"a}hler forms, Futaki invariants and convex geometry
  on toric Fano varieties}},
        date={1987},
     journal={Osaka Journal of Mathematics},
      volume={24},
      number={4},
       pages={705 \ndash  737},
}

\bib{mabuchi_mukai}{incollection}{
      author={Mabuchi, Toshiki},
      author={Mukai, Shigeru},
       title={Stability and {E}instein-{K}\"{a}hler metric of a quartic del
  {P}ezzo surface},
        date={1993},
   booktitle={Einstein metrics and {Y}ang-{M}ills connections ({S}anda, 1990)},
      series={Lecture Notes in Pure and Appl. Math.},
      volume={145},
   publisher={Dekker, New York},
       pages={133\ndash 160},
}

\bib{MM82}{article}{
      author={Mori, Shigefumi},
      author={Mukai, Shigeru},
       title={Classification of {F}ano {$3$}-folds with {$B\sb{2}\geq 2$}},
        date={1981/82},
     journal={Manuscripta Math.},
      volume={36},
      number={2},
       pages={147\ndash 162},
}

\bib{MM03}{article}{
      author={Mori, Shigefumi},
      author={Mukai, Shigeru},
       title={Erratum: ``{C}lassification of {F}ano 3-folds with {$B\sb 2\geq
  2$}'' [{M}anuscripta {M}ath. {\bf 36} (1981/82), no. 2, 147--162]},
        date={2003},
     journal={Manuscripta Math.},
      volume={110},
      number={3},
       pages={407},
}

\bib{GRDB-toricsmooth}{misc}{
      author={Obro, Mikkel},
       title={The {G}raded {R}ing {D}atabase},
         how={Database of smooth toric Fano varieties, online},
        note={\url{http://grdb.co.uk/forms/toricsmooth}},
}

\bib{odaka_spotti_sun}{article}{
      author={Odaka, Yuji},
      author={Spotti, Cristiano},
      author={Sun, Song},
       title={Compact moduli spaces of del {P}ezzo surfaces and
  {K}\"{a}hler-{E}instein metrics},
        date={2016},
        ISSN={0022-040X},
     journal={J. Differential Geom.},
      volume={102},
      number={1},
       pages={127\ndash 172},
         url={http://projecteuclid.org/euclid.jdg/1452002879},
}

\bib{petracci_survey}{incollection}{
      author={Petracci, Andrea},
       title={On deformations of toric {F}ano varieties},
        date={2022},
   booktitle={Interactions with lattice polytopes},
      series={Springer Proc. Math. Stat.},
      volume={386},
   publisher={Springer, Cham},
       pages={287\ndash 314},
}

\bib{cracked_fano_toric_ci}{article}{
      author={Prince, Thomas},
       title={Cracked polytopes and {F}ano toric complete intersections},
        date={2020},
     journal={Manuscripta Math.},
      volume={163},
      number={1-2},
       pages={165\ndash 183},
}

\bib{from_cracked_to_fano}{article}{
      author={Prince, Thomas},
       title={From cracked polytopes to {F}ano threefolds},
        date={2021},
     journal={Manuscripta Math.},
      volume={164},
      number={1-2},
       pages={267\ndash 320},
}

\bib{Suess}{article}{
      author={S{\"u}{\ss}, Hendrik},
       title={K{\"a}hler--einstein metrics on symmetric fano t-varieties},
        date={2013},
        ISSN={0001-8708},
     journal={Advances in Mathematics},
      volume={246},
       pages={100\ndash 113},
  url={https://www.sciencedirect.com/science/article/pii/S0001870813002326},
}

\bib{tian_del_pezzo}{article}{
      author={Tian, G.},
       title={On {C}alabi's conjecture for complex surfaces with positive first
  {C}hern class},
        date={1990},
        ISSN={0020-9910},
     journal={Invent. Math.},
      volume={101},
      number={1},
       pages={101\ndash 172},
}

\bib{tian_KE}{article}{
      author={Tian, Gang},
       title={K\"{a}hler-{E}instein metrics with positive scalar curvature},
        date={1997},
     journal={Invent. Math.},
      volume={130},
      number={1},
       pages={1\ndash 37},
}

\bib{tian_KstabKE}{article}{
      author={Tian, Gang},
       title={K-stability and {K}\"{a}hler-{E}instein metrics},
        date={2015},
        ISSN={0010-3640},
     journal={Comm. Pure Appl. Math.},
      volume={68},
      number={7},
       pages={1085\ndash 1156},
         url={https://doi.org/10.1002/cpa.21578},
}

\bib{xu_minimizing}{article}{
      author={Xu, Chenyang},
       title={A minimizing valuation is quasi-monomial},
        date={2020},
     journal={Ann. of Math. (2)},
      volume={191},
      number={3},
       pages={1003\ndash 1030},
}

\bib{xu_survey}{article}{
      author={Xu, Chenyang},
       title={K-stability of {F}ano varieties: an algebro-geometric approach},
        date={2021},
     journal={EMS Surv. Math. Sci.},
      volume={8},
      number={1-2},
       pages={265\ndash 354},
}

\bib{xu_zhuang}{article}{
      author={Xu, Chenyang},
      author={Zhuang, Ziquan},
       title={On positivity of the {CM} line bundle on {K}-moduli spaces},
        date={2020},
     journal={Ann. of Math. (2)},
      volume={192},
      number={3},
       pages={1005\ndash 1068},
}

\bib{zhuang}{article}{
      author={Zhuang, Ziquan},
       title={Birational superrigidity and {$K$}-stability of {F}ano complete
  intersections of index 1},
        date={2020},
        ISSN={0012-7094},
     journal={Duke Math. J.},
      volume={169},
      number={12},
       pages={2205\ndash 2229},
        note={With an appendix by Zhuang and Charlie Stibitz},
}

\end{biblist}
\end{bibdiv}

\end{document}